\documentclass[reqno,10pt]{amsart}

\usepackage{amsmath,amssymb,graphicx}

\usepackage[latin1]{inputenc}
\usepackage[english]{babel}

\newtheorem{Theorem}{Theorem}[section]
\newtheorem{proposition}[Theorem]{Proposition}
\newtheorem{corollary}[Theorem]{Corollary}
\newtheorem{lemma}[Theorem]{Lemma}

\theoremstyle{definition}
\newtheorem{Remark}[Theorem]{Remark}
\newtheorem{definition}[Theorem]{Definition}

\title{A Nekrasov--Okounkov type formula for $\widetilde{C}$}

\author[Mathias P\'etr\'eolle]{Mathias P\'etr\'eolle}
\address{Institut Camille Jordan, Universit\'e Claude Bernard Lyon 1,
69622 Villeurbanne Cedex, France}
\email{petreolle@math.univ-lyon1.fr}
\urladdr{http://math.univ-lyon1.fr/{\textasciitilde}petreolle}

\begin{document}
\maketitle

\begin{abstract}
In 2008, Han rediscovered an expansion of powers of Dedekind $\eta$ function attributed to Nekrasov and Okounkov  (which was actually first proved the same year by Westbury) by using a famous identity of Macdonald in the framework of type $\widetilde{A}$ affine root systems. In this paper, we obtain new combinatorial expansions of powers of $\eta$, in terms of partition hook lengths, by using the Macdonald identity in type $\widetilde{C}$ and a new bijection between vectors with integral coordinates and a subset of $t$-cores for integer partitions. As applications, we derive a symplectic hook formula and an unexpected relation between the Macdonald identities in types $\widetilde{C}$, $\widetilde{B}$, and $\widetilde{BC}$. We also generalize these expansions through the Littlewood decomposition and deduce in particular many new weighted generating functions for subsets of integer partitions and refinements of hook formulas.

\end{abstract}

\section{Introduction}
Recall the Dedekind $\eta$ function, which is a weight $1/2$ modular form defined as follows:
\begin{equation}
\eta(x) =x^{1/24} \prod _{k \geq 1} (1-x^k),
\end{equation}
where $x$ is a complex number satisfying $|x| < 1$ (we will assume this condition all along this paper). Apart from its modular properties, due to the factor $x^{1/24}$, this function plays a fundamental role in combinatorics, as it is related to the generating function of integer partitions. Studying expansions of powers of $\eta$ is natural, in the sense that it yields a certain amount of interesting questions both in combinatorics and number theory, such as Lehmer's famous conjecture (see for instance \cite{JPS}).
In 2006, in their study of the theory of Seiberg--Witten on supersymetric gauges in particle physics \cite{NO}, Nekrasov and Okounkov obtained the following formula:
\begin{equation}\label{nekrasov}
\prod_{k \geq 1} (1-x^k)^{z-1} = \sum_{\lambda \in \mathcal{P}} x^{|\lambda|}  \prod_{h\in \mathcal{H}(\lambda)} \left( 1-\frac{z}{h^2} \right),
\end{equation}
where $z$ is any complex number, $\mathcal{P}$ is the set of integer partitions and $\mathcal{H}(\lambda)$ is the multi-set of hook lengths of $\lambda$ (see Section~\ref{section1} for precise definitions). Actually, formula \eqref{nekrasov} was first proved algebraically through D'Arcais polynomials by Westbury in \cite{WEST}.
In 2008, this expansion was rediscovered and generalized by Han \cite{HAN}, through two main tools, one arising from an algebraic context and the other from a more combinatorial one. From this result, Han derived many applications in combinatorics and number theory, such as the marked hook formula, a reformulation of Lehmer's conjecture or a refinement of a result due to Kostant \cite{KOS}. Formula \eqref{nekrasov} was next proved and generalized differently by Iqbal \emph{et al.} in 2013 \cite{IEA} by using plane partitions, the Cauchy formula for Schur functions and the notion of topological vertex. 

Apart from the many applications given by Han in \cite{ HAN09,HAN, HAN10b}, and Han and Ji in \cite{HJ}, formula \eqref{nekrasov} was also used in~\cite{CW} to derive new Ramanujan-type congruences over small arithmetic progressions and in \cite{KEI} to obtain polynomial analogues of Ramanujan congruences. It  was also used in \cite{VEL} to give an exact expression of some specific interesting coefficients through analytic number theoretic methods.

Han's proof of \eqref{nekrasov} uses on the one hand a bijection between $t$-cores and some vectors of integers, due to Garvan, Kim and Stanton in their proof of Ramanujan congruences \cite{GKS}. Recall that $t$-cores had originally been introduced in representation theory to study some characters of the symmetric group \cite{GK}. On the other hand, Han uses the Macdonald identity for affine root systems \cite{ARS}. Recall that it is a generalization of Weyl formula for finite root systems $R$ which can itself be written as follows:
\begin{equation}
\prod _{\alpha >0} (e^{\alpha/2} - e^{-\alpha/2}) = \sum_{w \in W(R)} \varepsilon (w) e^{w \rho},
\end{equation}
where the sum is over the elements of the Weyl group $W(R)$, $\varepsilon$ is the sign, and $\rho$ is an explicit vector depending on $W(R)$. Here, the product ranges over the positive roots $R^+$, and the exponential is formal. Macdonald specialized his formula for several affine root systems and exponentials. In all cases, when the formal exponential is specialized to the constant function equal to 1, the product side can be rewritten as an integral power of Dedekind $\eta$ function (this power is the dimension of a compact Lie group having $R$ as its system of roots). In particular, the specialization used in~\cite{HAN} corresponds  to the type $\widetilde{A}_t$, for an odd positive integer $t$, and reads (here $\|.\|$ is the euclidean norm):
\begin{equation}\label{equaA}
\eta(x)^{t^2-1} =c_0 \sum_{{\bf v}} x^{\|{\bf v}\|^2 /2t}\prod_{i<j}(v_i-v_j) , \quad \mbox{with~} c_0:= \frac{(-1)^{(t-1)/2}}{1! 2! \cdots (t-1)!},
\end{equation}
where the sum is over $t$-tuples ${\bf v} :=(v_0,\ldots, v_{t-1}) \in \mathbb{Z}^{t}$ such that $v_i \equiv i \mbox{~mod~} t$ and $v_0+\cdots +v_{t-1}=0$. Han next uses a refinement of the aforementioned bijection from \cite{GKS} to transform the right-hand side into a sum over partitions,  and proves \eqref{nekrasov} for all odd integers $t$. He finally transforms the right-hand side through very technical considerations to show that \eqref{nekrasov} is in fact true for all complex numbers $t$. A striking remark is that the factor of modularity $x^{(t^2-1)/24}$ in $\eta(x)^{t^2-1}$ cancels naturally in the proof when the bijection is used.
\medskip

This approach immediately raises a question, which was asked by Han in~\cite[Problem 6.4]{HAN09}: can we use specializations of the Macdonald formula for other types to find new combinatorial expansions of the powers of $\eta$? In the present paper, we give a positive answer for the case of type  $\widetilde{C}$ and, as shall be seen later, for types $\widetilde{B}$ and $\widetilde{BC}$. In type $\widetilde{C}_t$, $t \geq 2$ being an integer, the Macdonald formula reads:
\begin{equation}\label{equaC}
\eta(x)^{2t^2+t} = c_1 \sum_{{\bf v}} x^{\|{\bf v}\|^ 2/(4t+4)} \prod_i v_i \prod_{i<j}(v_i^2-v_j^2),
\end{equation}
where $\displaystyle c_1:= \frac{(-1)^{\lfloor t/2 \rfloor}}{1! 3! \cdots (2t+1)!}$, and the sum ranges over $t$-tuples ${\bf  v} :=(v_1,\ldots, v_t) \in \mathbb{Z}^t$ such that $v_i \equiv i \mbox{~mod~} 2t+2$. The first difficulty in providing an analogue of \eqref{nekrasov} through \eqref{equaC} is to find which combinatorial objects should play the role of the partitions $\lambda$. Our main result is the following possible answer.

\begin{Theorem}\label{theoremeintro} For any complex number $t$, with the notations and definitions of Section~\ref{section1}, the following expansion holds:
\begin{equation}\label{eqtheoremeintro}
\prod_{k \geq 1}(1-x^k) ^{2t^2+t} = \sum_{\lambda \in DD}\delta_\lambda\,   x^{|\lambda|/2} \prod_{h \in  \mathcal{H}(\lambda)} \left( 1- \frac{2t+2}{h\, \varepsilon_h }\right),
\end{equation}
where the sum is over doubled distinct partitions $\lambda$, $\delta_\lambda$ is equal to $1$ (respectively $-1$) if the Durfee square of $\lambda$ is  of even (respectively odd) size, and $\varepsilon_h$ is equal to $-1$ if $h$ is the hook length of a box strictly above the diagonal in the Ferrers diagram of $\lambda$ and to $1$ otherwise.
\end{Theorem}
 The global strategy to prove this is to use~\eqref{equaC} and a bijection, obtained through results of~\cite{GKS}, between the set of vectors involved in \eqref{equaC} and pairs of  partitions $(\lambda,\mu)$, where $\lambda$ is a self-conjugate $t+1$-core and $\mu$ is a doubled distinct $t+1$-core (precise definitions are given in Section~\ref{section1}). Some technical lemmas regarding $2t+2$-compact sets and the principal hook lengths of $(\lambda,\mu)$ allow us to prove Theorem~\ref{thmprincipal} below (an argument of polynomiality is also needed). A bijection between pairs $(\lambda,\mu)$ and doubled distinct partitions is then used to establish Theorem~\ref{theoremeintro}.
 
Many applications can be derived directly from Theorem~\ref{theoremeintro}. However, we will only highlight three of them here. The first is the following symplectic analogue (corresponding to type $ B$) of the famous combinatorial hook formula (which has many consequences in representation theory, see for instance \cite[Chapter 7]{EC}), valid for any positive integer $n$:
\begin{equation}\label{hookf}
\sum_{\stackrel{\lambda \in DD}{  |\lambda|= 2n} } \prod_{h \in \mathcal{H}(\lambda)}\frac{1}{h} = \frac{1}{2^n n!}.
\end{equation}

The second, which is more algebraic and expressed in Theorem~\ref{equi} below, is a surprising link between the family of Macdonald formulas in types $\widetilde{C}_t$ (for all integers $t \geq 2$), the one in types $\widetilde{B}_t$ (for all integers $t \geq 3$), and the one in types $\widetilde{BC}_t$ (for all integers $t \geq 1$). 

The third (see Theorem~\ref{kostantpet} below) is an improvement of a result due to Kostant \cite{KOS}, related to classical number theoretic questions regarding the non-nullity of some coefficients $f_k(s)$, where we write:
\begin{equation*}
\prod_{n \geq 1} (1-x^n)^s= \sum_{k\geq 0} f_k(s) x^k.
\end{equation*}
\smallskip

Nevertheless, many more applications can be derived when one aims to refine Theorem \ref{theoremeintro}, by adding more parameters, as did Han for \eqref{nekrasov}. We can prove the following result, by using the canonical correspondence between partitions and bi-infinite binary words beginning with infinitely many $0$'s and ending with infinitely many $1$'s, together with the classical Littlewood decomposition (see for instance \cite[p. 468]{EC}) and new properties for it. More precisely, we  use the stricking fact, specific to doubled distinct partitions, that the statistics $h\varepsilon_h$ and $\delta_\lambda$ of a doubled distinct partition $\lambda$ can be expressed in terms of the image of $\lambda$ under the Littlewood decomposition.
\begin{Theorem}\label{generalisation}
Let $t=2t'+1$ be an odd positive integer.  With the notations and definitions of Section~\ref{section1}, the following equality holds:
\begin{multline}\label{eqgeneralisation}
\sum_{\lambda \in DD} \delta_{\lambda}\, x^{|\lambda|/2} \prod_{ h \in \mathcal{H}_t(\lambda)}\left( y -\frac{yt(2z+2)}{\varepsilon_h ~ h}\right)\\= \prod_{k \geq 1}  (1-x ^k)(1-x^{kt})^ {t'-1} \left(1-x ^{kt}y^{2k}\right)^{(2z+1)(zt+3(t-1)/2)},
\end{multline}
where the sum ranges over doubled distinct partitions, and $\mathcal{H}_t(\lambda)$ is the multi-set of hook lengths of $\lambda$ which are integral multiples of $t$.
\end{Theorem}
 As shall be seen later, we derive several specializations of \eqref{eqgeneralisation}, such as many weighted generating functions for doubled distinct partitions and the following generalization of the symplectic hook formula \eqref{hookf}, valid for any positive integer $n$ and any odd positive integer $t$:
\begin{equation}\label{hookfgen}
 \sum_{\stackrel{ \lambda \in DD, ~ |\lambda|=2t n}{\# \mathcal{H}_t(\lambda)=2n }}\displaystyle\;\delta_\lambda \prod _{h \in \mathcal{H}_t (\lambda)} \frac{1}{ h\, \varepsilon_h}=\displaystyle \frac{(-1)^n}{n! t^n 2^n}.
\end{equation}
Note that it is not completely trivial that \eqref{hookf} is a consequence of \eqref{hookfgen}; the details are given in Section~\ref{section4}.
\medskip

\begin{Remark}\label{king}
During the presentation of a part of this work at the FPSAC 2015  conference in Daejeon, South Korea, we had fruitful discussions with Ronald King and Bruce Westbury. From these, it appeared that Theorem~\ref{theoremeintro}  and the Nekrasov--Okounkov formula~\eqref{nekrasov} can actually  be proved in a purely algebraic manner, through both a result of King and a result due to El Samra--King. (As said before, we also 
mention that~\eqref{nekrasov} already appeared in~\cite{WEST}, here again the proof is algebraic and uses D'Arcais polynomials.) We quickly outline the proof in type $\widetilde{C}$, as was communicated to us by King in \cite{KING15}. By using a modification rule for the characters of Lie algebras, the following equation is proved in~\cite[Equation (5.8b)]{KING89}:
\begin{equation}\label{eqking}
C_{q^k}(x)_N:=\prod_{1 \leq i\leq j \leq N} (1-q^kx_ix_j)= \sum_{\gamma \in DD} (-q)^{|\gamma|/2} \langle\gamma\rangle(x)_N,
\end{equation}
where $N$ is a positive integer, and $\langle\gamma\rangle(x)_N$ is the formal irreducible character of the symplectic group $Sp(N)$ associated with the partition $\gamma$.
By specializing the variables $x_i$ to $1$, and computing the limit, one obtains:
\begin{equation}
\prod_{k =1}^\infty (1-q^k)^{N(N+1)/2} = \sum_{\gamma \in DD} (-q)^{|\gamma|/2} D_N\langle\gamma\rangle,
\end{equation}
where $ D_N\langle\gamma\rangle$ is the formal dimension of the irreducible representation of $Sp(N)$ associated with $\gamma$. This dimension is computed in \cite[Equation 3.29]{KS}, which yields  Theorem~\ref{theoremeintro} after some calculations. Nevertheless, coming back to the right-hand side of Macdonald formula~\eqref{equaC} algebraically seems, according to King, quite tricky (as the vectors in Macdonald identities  are not in the Weyl dominant sector). Our approach therefore gives not only Theorem~\ref{theoremeintro}, but also the combinatorics which explains the link between the vectors involved in \eqref{equaC} and integers partitions. Moreover, Theorem~\ref{generalisation} seems out of reach through King's algebraic point of view.
\end{Remark}

This paper is organized as follows. In Section~\ref{section1}, we recall the definitions and notations regarding partitions, $t$-cores, self-conjugate and doubled distinct partitions. Section~\ref{section2} presents bijections between the already mentioned subfamilies of partitions and some vectors of $\mathbb{Z}^{t}$, and theirs properties that we will explain.  More precisely, Section~\ref{section3.1} introduces a new bijection between the vectors of integers involved in \eqref{equaC} and the pairs of self-conjugate and doubled distinct $t+1$-cores, Sections~\ref{section3.2}--\ref{section3.3} are devoted to the proof of Theorem~\ref{theoremeintro}, while Section~\ref{section3.4} presents the symplectic hook formula \eqref{hookf}, and the connection between \eqref{equaC} and the Macdonald identities in types $\widetilde{B}$ and $\widetilde{BC}$, which are shown in Theorem~\ref{equi} to be all generalized by Theorem~\ref{theoremeintro}. The aforementioned improvement of a result due to Kostant is given in Section~\ref{section3.5}. In Section~\ref{section4}, we prove Theorem~\ref{generalisation}. More accurately, Section~\ref{section4.1} recalls the Littlewood decomposition, while Section~\ref{section4.2} shows new properties for it concerning doubled distinct partitions. Section~\ref{section4.3} is devoted to the proof of Theorem~\ref{generalisation}, from which we deduce many applications. We end by some questions and remarks in Section~\ref{section5}.

\section{Integer partitions and $t$-cores}\label{section1}
In all this section, $t$ is a fixed positive integer.
\subsection{Definitions}\label{defs}
We recall the following definitions, which can be found in \cite{EC}. A \emph{partition} $\lambda=(\lambda_1,\lambda_2,\ldots, \lambda_\ell)$ of the integer $n \geq 0$ is a finite non-increasing sequence of positive integers whose sum is $n$. The $\lambda_i$'s are the \textit{parts} of $\lambda$, $\ell := \ell(\lambda)$ is its \textit{length}, and $n$ its \textit{weight}, denoted by $|\lambda|$. Each partition can be represented by its \textit{Ferrers diagram} as shown in Figure~\ref{fig1.1}, left. (Here we represent the Ferrers diagram in French convention.)

\begin{figure}[!h] \begin{center}
\includegraphics[scale=1.2]{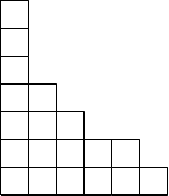}
\hspace*{1cm}
\includegraphics[scale= 1.2]{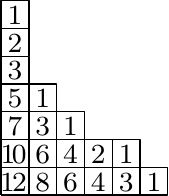}
\hspace*{1cm}
\includegraphics[scale=1.2]{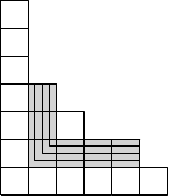}
\caption{\label{fig1.1} The Ferrers diagram of the partition $(6,5,3,2,1,1,1)$, the hook lengths, and a principal hook.}
\end{center}
\end{figure}

For each box $v=(i,j)$ in the Ferrers diagram of $\lambda$ (with $i \in \{1,\ldots, \ell\}$ and $j \in \{1,\ldots,\lambda_i \}$), we define the \emph{hook of $v$} as the set of boxes $u$ such that either $u$ lies on the same row and above $v$, or $u$ lies on the same column and on the right of $v$. The \emph{hook length} $h_v$ of $v$ is the cardinality of its hook (see Figure~\ref{fig1.1}, center).
The hook of $v$ is called \emph{principal} if $v=(i,i)$ (that is  $v$ lies on the diagonal of $\lambda$, see Figure~\ref{fig1.1}, right). The \emph{Durfee square} of $\lambda$ is the greatest square included in its Ferrers diagram, the length of its side is the \emph{Durfee length}, denoted by $D(\lambda)$: it is also the number of principal hooks. We denote by $\delta_\lambda$ the number $(-1)^{D(\lambda)}$. The \emph{hook length multi-set} of $\lambda$, denoted by $\mathcal{H}(\lambda)$, is the multi-set of all hook lengths of $\lambda$.
\begin{definition}Let $\lambda $ be a partition. We say that $\lambda$ is a \emph{ $t$-core} if and only if no hook length of $\lambda$ is a multiple of $t$. 
\end{definition}
Recall \cite[p. 468]{EC} that $\lambda$ is a $t$-core if and only if the hook length multi-set of $\lambda$ does not contain the integer $t$. We denote by $\mathcal{P}$ the set of partitions and by  $\mathcal{P}_{(t)}$ the subset of $t$-cores. We also recall that the definition of ribbons can be found in \cite{EC} (see Figure~\ref{fig2} for an example).
\begin{definition}Let  $\lambda $ be a partition. The \emph{$t$-core of $\lambda$} is the partition $T(\lambda)$ obtained from $\lambda$  by removing in its Ferrers diagram all the ribbons of length $t$, and by repeating this operation until we can not remove any such ribbon. 
\end{definition}

\begin{figure}[h!]
\begin{center}
\includegraphics[scale=1.2]{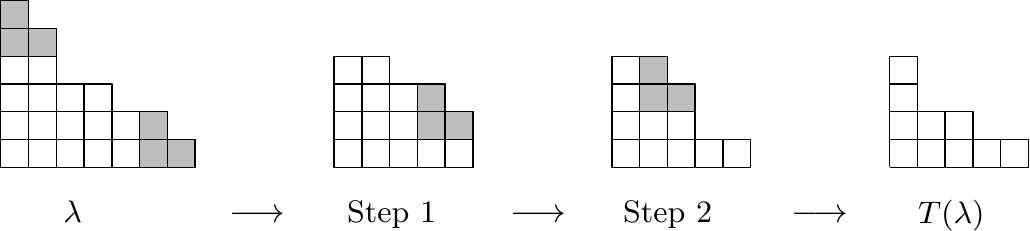}
\caption{\label{fig2} The construction of the $3$-core of the partition $\lambda=(7,6,4,2,2,1)$. In grey, the deleted ribbons.}
\end{center}
\end{figure}

 Note that $T(\lambda)$ does not depend on the order of removal (see \cite[p. 468]{EC} for a proof). In particular, as a ribbon of length $t$ corresponds bijectively to a box with hook length $t$, the $t$-core $T(\lambda)$ of a partition $\lambda$ is itself a $t$-core.
\subsection{A bijection between $t$-cores and vectors of integers}\label{bijphi}
We will need later restrictions of a bijection from \cite{GKS} to two subsets of $t$-cores. Here, we first recall this bijection. Let $\lambda$ be a $t$-core, we define the vector $ \phi(\lambda):=(n_0, n_1,\ldots, n_{t-1})$ as follows. We label the box $(i,j)$ of $\lambda$ by $(j-i)$ modulo $t$. We also label the boxes in the (infinite) column 0 in the same way, and we call the resulting diagram the \emph{extended t-residue diagram} (see Figure~\ref{fig3} below). A box is called \emph{exposed} if it is at the end of a row of the extended $t$-residue diagram. The set of boxes $(i,j)$ of the extended $t$-residue diagram satisfying $t(r-1)\leq j-i<tr$ is called a \emph{region} and labelled $r$. We define $n_i$ as the greatest integer $r$ such that the region labelled $r$ contains an exposed box with label $i$. 
\begin{figure}[h!]\begin{center}
\includegraphics[scale=1]{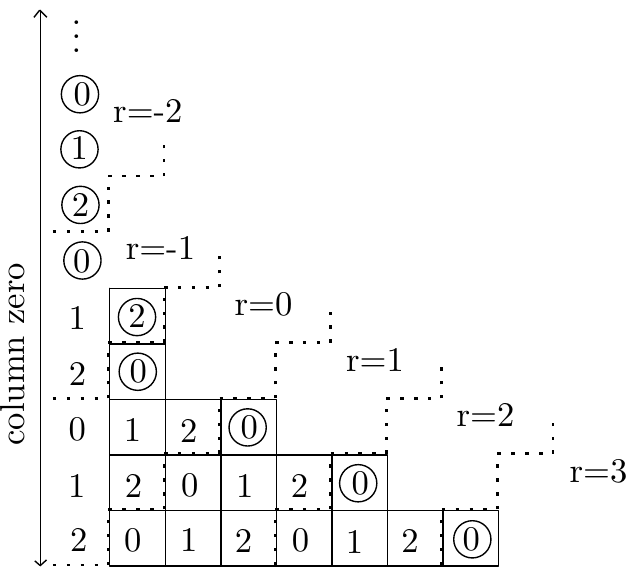}
\caption{\label{fig3}The extended 3-residue diagram of the $3$-core $\lambda=(7,5,3,1,1)$. The exposed boxes are circled.}
\end{center}
\end{figure}
\begin{Theorem}[\cite{GKS}]\label{GKS}
The map $\phi$ is a bijection between $t$-cores and vectors of integers $ {\bf n}=(n_0,n_1,\ldots, n_{t-1}) \in \mathbb{Z}^t$, satisfying $n_0+ \cdots +n_{t-1}=0$, such that: \begin{equation} |\lambda | = \frac{t\|{\bf n}\|^2}2 + {\bf b \cdot n}= \frac{t}{2}\sum_{i=0}^{t-1}n_i^2+ \sum_{i=0}^{t-1}in_i,\end{equation}
where~ ${\bf b } := (0,1,\ldots,t-1)$, $\|{\bf n}\|$ is the euclidean norm of ${\bf n}$, and ${\bf b \cdot n}$ is the scalar product of ${\bf b}$ and ${\bf n}$.
\end{Theorem}

For example,  the $3$-core $\lambda=(7,5,3,1,1)$ of Figure~\ref{fig3} satisfies $\phi(\lambda)$ = $(3,-2,-1)$. We indeed have $7+5+3+1+1=17= |\lambda|= \frac{3}{2}(9+4+1)-2-2.$

\subsection{Self-conjugate $t$-cores}
Next we come to the definition of a subfamily of $\mathcal{P}_{(t)}$ which naturally appears in the proof of our type $\widetilde{C}$ formula. We define \emph{self-conjugate partitions} (respectively \emph{self-conjugate $t$-cores}) as elements $\lambda$ in $\mathcal{P}$ (respectively  $\mathcal{P}_{(t)}$) satisfying $\lambda=\lambda^*$, where $\lambda^*$ is the conjugate of $\lambda$ (see \cite[p. 287]{EC}).  We denote by $SC$ the set of self-conjugate partitions and by $SC_{(t)}$ its subset of $t$-cores. We also denote by $\lfloor t/2 \rfloor$ the greatest integer smaller or equal to $t/2$. 

\begin{figure}[h!]
\includegraphics[scale=1.2]{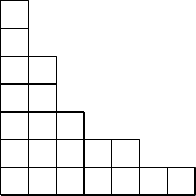}
\caption{\label{fig1}The self-conjugate partition $\lambda=(7,5,3,2,2,1,1)$.}
\end{figure}

\begin{proposition}\label{autoconjugue} There is a bijection $\phi_1$ between the partitions $\lambda \in SC_{(t)}$ and vectors of integers $ \phi_1(\lambda):={\bf n}\in \mathbb{Z}^{\lfloor t/2 \rfloor},$ such that: 
\begin{equation} \label{eqlambda}|\lambda | = t\|{\bf n}\|^2 + {\bf c} \cdot {\bf n} , \quad \mbox{with~}
{\bf c } :=  \left\lbrace \begin{array}{lr}
    (1,3,\ldots,t-1) &\mbox{for $t$ even}, 
\vspace*{0.05cm}\\ (2,4,\ldots,t-1)&\mbox{for $t$ odd}.
  \end{array}\right. \end{equation}
\end{proposition}

This result is a direct consequence of \cite[Equation (7.4)]{GKS}, by defining the image of a self-conjugate $t$-core $\lambda$ under $\phi_1$ as the vector whose components are the $\lfloor t/2 \rfloor$  last ones of $\phi(\lambda)$. 

For example, the self-conjugate $3$-core $\lambda$ of Figure~\ref{fig1} satisfies $\phi(\lambda)=(3,0,-3)$; therefore its image under $\phi_1$ is the vector $(-3)$.

\subsection{$t$-cores of doubled distinct partitions}\label{propdd} We will also need a second subfamily of $\mathcal{P}_{(t)}$ in our proof of Theorem~\ref{theoremeintro}. Let $\mu^0 $ be a partition with distinct parts. We denote by S($\mu^0$) the shifted Ferrers diagram of $\mu^0$, which is its Ferrers diagram where for all $1\leq i \leq \ell(\mu^0)$, the $i^{th}$ row is shifted by $i$ to the right (see Figure~\ref{fig4} below). 

\begin{definition}[\cite{GKS}] We define the \emph{doubled distinct partition} $\mu$ of $\mu^0$ as the partition whose Ferrers diagram is obtained by adding $\mu^0_i$ boxes to the  $i^{th}$ column of S($\mu^0$) for all $1\leq i \leq \ell(\mu^0)$. We denote by $DD$ the set of doubled distinct partitions and by $DD_{(t)}$ its subset of $t$-cores. 
\end{definition}

\begin{figure}[h!]\begin{center}
\includegraphics[scale=1.2]{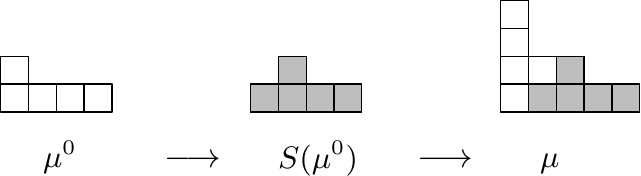}
\caption{\label{fig4}The construction of the doubled distinct partition $\mu =(5,3,1,1)$, for  $\mu^0=(4,1)$.}\end{center}
\end{figure}

Notice that the set $DD$ can also be defined as the set of partitions $\mu =(\mu_1, \ldots, \mu_\ell)$ such that $\mu_i = \mu_i^*+1$ for all $i \in \{1,\ldots,  D(\mu)\}$. Moreover, by definition and direct computations, a doubled distinct partition $\mu$ satisfies the following properties, illustrated in Figure~\ref{fig5.0}, and which will be useful later in Theorem~\ref{bijcouple} and Lemma~\ref{littlewood}:
\begin{itemize}
\item if $1\leq i, j \leq D(\mu)$, then the boxes $(i,j)$ and $(j,i)$ have the same hook length,
\item if $1\leq i \leq D(\mu)$, then the hook length of the box $(i,i)$ is twice the hook length of $(i, D(\mu)+1)$,
\item if $D(\mu)+1 \leq i$ and $1\leq j \leq D(\mu)$, then the boxes $(i,j)$ and $(j, i+1)$ have the same hook length.
\end{itemize}
\begin{figure}[h!]
\includegraphics[scale=1.2]{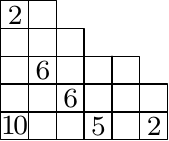}
\caption{\label{fig5.0}Illustration for $\mu=(6,6,5,3,2)$.}
\end{figure}

\begin{proposition}\label{DD}There is a bijection $\phi_2$ between the partitions $\mu \in DD_{(t)}$ and vectors of integers $ \phi_2(\mu) :={\bf n}\in \mathbb{Z}^{\lfloor (t-1)/2 \rfloor},$ such that:  \begin{equation}\label{eqmu} |\mu | =t\|{\bf n}\|^2 + {\bf d}\cdot {\bf n} ,
\quad \mbox{with~} {\bf d}:=\left\lbrace\begin{array}{ll}
(2,4,\ldots,t-2)&\mbox{for $t$ even},\vspace*{0.05cm}
\\ (1,3,\ldots,t-2)&\mbox{for $t$ odd}.\end{array}\right.\end{equation}
\end{proposition}
Again, Proposition~\ref{DD} is a direct consequence of~\cite[Bijection 4]{GKS}, by defining the image of a doubled distinct $t$-core $\mu$ under $\phi_2$ as the vector whose components are the $\lfloor (t-1)/2 \rfloor$ last ones of $\phi(\mu)$. Equation \eqref{eqmu} comes from the relationships $n_0=0$, and $n_i=n_{t-i}$ for $1\leq i \leq t-1$ in~\cite[Bijection 4]{GKS}.

For example, the doubled distinct $3$-core  $\mu=(5,3,1,1)$ of Figure~\ref{fig4}, right,  satisfies $\phi(\mu)=(0,2,-2)$; so its image under $\phi_2$ is the vector $(-2)$.

\subsection{Generating function of $SC_{(t)} \times DD_{(t)}$}
We will now focus on pairs of $t$-cores in the set $SC_{(t)} \times DD_{(t)}$. We can in particular compute the generating function of these objects. Let  $(\lambda,\mu)$ be an element of $SC_{(t)} \times DD_{(t)}$. We define the weight of $(\lambda,\mu)$ as
$ |\lambda|+|\mu|$, and we denote by $ h_t$ the generating function
\begin{equation}h_t(q) :=\sum_{(\lambda,\mu) \in SC_{(t)} \times DD_{(t)}} q^{|\lambda|+|\mu|}.\end{equation}

We would like to mention that the first step towards discovering Theorem~\ref{theoremeintro} was the computation of the Taylor expansion of $h_3(q)$, whose first terms seemed to coincide with the ones in the generating function of the vectors of integers involved in $\eqref{equaC}$ for $t=2$.
\begin{proposition}\label{propgeneratingfunction}
The following equality holds for any integer $t\geq 1$ and any complex number $q$ satisfying $|q|<1$:
\begin{equation}\label{generatingfunction}
h_{t}(q) =\frac{(q^2;q^2)_\infty}{(q;q)_\infty} (q^{t};q^{t})_\infty (q^{2t};q^{2t})^{t-2}_\infty,
\end{equation}
where $\displaystyle(a;q)_\infty := \prod_{j \geq 1} (1-aq^{j-1})$ is the usual infinite $q$-rising factorial.
\end{proposition}
\begin{proof} The generating functions of $SC_{(t)}$ and $DD_{(t)}$ are already known (see~\cite{GKS}), but both have a different expression according to the parity of $t$. We assume that $t$ is odd.
The generating function of self-conjugate $t$-cores is then:
\begin{equation*}
\sum_{\lambda \in SC_{(t)}} q^{|\lambda|}= \frac{(-q;q^2)_\infty(q^{2t};q^{2t})^{(t-1)/2}_\infty}{(-q^{t};q^{2t})_\infty},
\end{equation*}
while the generating function of doubled distinct $t$-cores is
\begin{equation*}
\sum_{\mu \in DD_{(t)}} q^{|\mu|}= \frac{(-q^2;q^2)_\infty(q^{2t};q^{2t})^{(t-1)/2}_\infty}{(-q^{2t};q^{2t})_\infty}.
\end{equation*}
Thanks to the classical equality $(-q^{t};q^{2t})_\infty(-q^{2t};q^{2t})_\infty=(-q^{t};q^{t})_\infty$, multiplying these two generating functions gives:
\begin{eqnarray*}
\displaystyle h_{t}(q) &=(-q;q)_\infty  \frac{\displaystyle(q^{2t};q^{2t})_\infty}{\displaystyle(-q^{t};q^{t})_\infty}(q^{2t};q^{2t})^{t-2}_\infty
\\&=\frac{\displaystyle(q^2;q^2)_\infty}{\displaystyle(q;q)_\infty} (q^{t};q^{t})_\infty (q^{2t};q^{2t})^{t-2}_\infty.
\end{eqnarray*}

As the computation of the case $t$ even is almost identical, we do not write it here.
  \end{proof}

\section{A Nekrasov-Okounkov type formula in type $\widetilde{C}$}\label{section2} 
The goal of this section is to prove Theorem~\ref{theoremeintro}. The global strategy is the following: we start from the Macdonald formula \eqref{equaC} in type $\widetilde{C}_t$, in which we replace the sum over vectors of integers by a sum over pairs of $t+1$-cores, the first in $SC_{(t+1)}$, and the second in $DD_{(t+1)}$. To do this, we need a new bijection $\varphi$ satisfying some properties that we will explain. This will allow us to establish Theorem~\ref{thmprincipal} of Section~\ref{section3.2} below for all integers $t \geq 2$. An argument of polynomiality will then enable us to extend this theorem to any complex number $t$. Then, a natural bijection between pairs $(\lambda,\mu)$ in $SC \times DD$, and doubled distinct partitions (with weight equal to $|\lambda|+|\mu|$) will allow us to conclude. Note that at this final step, the partitions need not be $t+1$-cores any more.

\subsection{The bijection $\varphi$}\label{section3.1}
In what follows, we assume that $t\geq 2$ is an integer.
\begin{definition}\label{deltai} If  $(\lambda, \mu)$ is a pair belonging to $SC_{(t+1)} \times DD_{(t+1)}$, we denote by $\Delta$ the \emph{set of principal hook lengths} of $\lambda$ and $\mu$, and for all $ i \in \{1,\ldots, t\}$ we define
\begin{equation}\Delta_i:= Max\left(\{ h \in \Delta , h \equiv \pm i -t-1 \mbox{~mod~}2t+2 \} \cup \{i-t-1\}\right).
\end{equation}
\end{definition}
For example, for $\lambda= (7,5,3,2,2,1,1)$, $\mu = (5,3,1,1)$ and $t+1=3$, we have $\Delta=\{13,8,7,2,1\}$, $\Delta_1=8$, and $
\Delta_2 =13$ (see Figure~\ref{fig5}).
\begin{figure}[h!]\begin{center}
\includegraphics[scale=1.2]{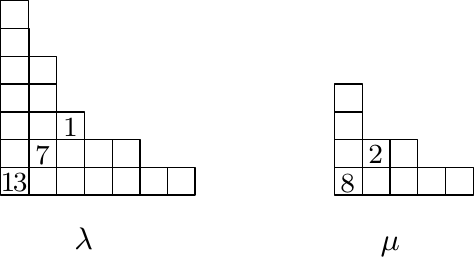}
\caption{\label{fig5}Computation of $\Delta$, $\Delta_1$ and $\Delta_2$ for a $(\lambda, \mu) \in SC_{(3)} \times DD_{(3)}$.}\end{center}
\end{figure}

As $\lambda$ (respectively $\mu$) is self-conjugate (respectively doubled distinct), all of its principal hook lengths are odd (respectively even).  The knowledge of the set $\Delta$ enables us to reconstruct uniquely both partitions $\lambda$ and $\mu$. The following theorem shows that in fact, when these two partitions are $t+1$-cores, it is enough to know the numbers $\Delta_i$ to recover $\lambda$ and $\mu$ (so knowing the hook length maxima in each congruency class modulo $2t+2$ is enough).

Recall the bijections $\phi_1$ and $\phi_2$ defined in Propositions~\ref{autoconjugue} and \ref{DD}, and set $(\lambda, \mu)$ a pair belonging to $SC_{(t+1)} \times DD_{(t+1)}$. We define $\varphi(\lambda,\mu):={\bf n}=(n_1,\ldots, n_t) \in \mathbb{Z}^t $ as follows:
 \begin{itemize}
 \item if $(t+1)$ is odd, then $n_{2i}$ (respectively  $n_{2i+1}$) is the $i^{th}$ component of $\phi_1(\lambda)$ (respectively $\phi_2(\mu)$);
 \item if $(t+1)$ is even, $n_{2i}$ (respectively $n_{2i+1}$) is the $i^{th}$component of $\phi_2(\mu)$ (respectively $\phi_1(\lambda)$) .
 \end{itemize}
 
\begin{Theorem}\label{phi} Let $t\geq 2$ be an integer and set ${\bf e}:=(1,2,\ldots, t)$. 
The map $\varphi$ between $SC_{(t+1)} \times DD_{(t+1)}$ and $ \mathbb{Z}^t$ is a bijection such that $\varphi(\lambda,\mu):={\bf n}=(n_1,\ldots, n_t)$ satisfies:
\begin{equation}\label{eqlambdamu} |\lambda|+|\mu|= (t+1) \|{\bf n}\| ^2 +{\bf e} \cdot {\bf n} .\end{equation}
Besides, the following relation holds for all integers $i \in \{1,\ldots, t\}$: 
\begin{equation}\label{liennideltai}t+1+\Delta_i= \sigma_i ((2t+2)n_i+i),\end{equation}
where $\sigma_i$ is equal to $1 $ (respectively $-1$) if $n_i \geq 0$ (respectively $n_i<0$).
\end{Theorem}

For example, the pair of $3$-cores  ($\lambda,\mu$) of Figure~\ref{fig5} satisfies $\varphi(\lambda,\mu)=(-2,-3)$. We have $31 =  |\lambda|+|\mu| =3(4+9)+1(-2)+2(-3)$.
Moreover, $\Delta_1=8$, $
\Delta_2 =13$. We verify that $3+\Delta_1=11=-\left(6n_1+1\right)$, and
$3+\Delta_2=16=-\left(6n_2+2 \right).$

\begin{proof}
From its definition, it is obvious that $\varphi$ is a bijection, as the concatenation of two bijections. To establish \eqref{eqlambdamu}, it is sufficient to sum \eqref{eqlambda} and \eqref{eqmu} in which we replace $t$ by $t+1$. It only remains to prove \eqref{liennideltai}, which is the difficult point of our theorem.

Let $(\lambda, \mu)$ be a pair in $SC_{(t+1)}\times DD_{(t+1)}$, set ${\bf n}=(n_1,\ldots, n_t)$ its image under $\varphi$, and $i \in \{1,\ldots, t\}$. We denote by  ${\bf n'}=(n_0',n_1',\ldots, n_t')$ (respectively ${\bf n''}=(n_0'',n_1'',\ldots, n_t'')$)  the image of $\lambda$ (respectively $\mu$) under $\phi$, and  by $h_{\lambda, \ell\ell}$ (respectively $h_{\mu, \ell \ell}$) the hook length of the box ($\ell, \ell$) of $\lambda$ (respectively $\mu$).
The proof of \eqref{liennideltai} comes from a precise examination of the action of the bijections $\phi$, $\phi_1$ and $\phi_2$ on $\lambda$ and $\mu$. Twelve cases occur, according to the following distinctions:
\begin{itemize}
\item $(t+1)$ is even or odd,
\item $i$ is even or odd,
\item $n_i$ is positive, negative, or equal to $0$.
\end{itemize}

{\bf Case $1$:} $t+1$ is even, $i=2j+1$ is odd, $n_{2j+1} >0$.
\\ In this case, the definition of $\varphi$  guarantees that $n_{2j+1}=n_{\frac{t+1}{2}+j+1}'$. By definition of $\phi$ (see Section~\ref{bijphi}), the positivity of $n_{\frac{t+1}{2}+j+1}'$ implies that there is in the extended $(t+1)$-residue diagram of $\lambda$ a row with minimal index $\ell$ (with $\ell \leq D(\lambda)$), such that $\frac{t+1}{2}+j$ is exposed at the end of this row, in the region $n_{\frac{t+1}{2}+j+1}'$.  As $\lambda$ is self-conjugate, we can deduce that:
\begin{align*}
h_{\lambda,\ell\ell}=& 2\left[(t+1)(n_{\frac{t+1}{2}+j+1}' -1)+\frac{t+1}{2}+j+1 \right]-1
\\=&(2t+2)n_{2j+1}-t-1+(2j+1).
\end{align*}
In particular, we have $h_{\lambda,\ell\ell} \equiv 2j+1 +t+1\mbox{~mod~}2t+2$, and by minimality of $\ell$, this principal hook length is maximal in its congruency class. So $h_{\lambda,\ell\ell}=\Delta_{2j+1}$ and therefore
\begin{equation*}\Delta_{2j+1}=\left[(2t+2)n_{2j+1}+(2j+1)\right]-t-1.\end{equation*}

{\bf Case $2$:} $t+1$ is even, $i=2j+1$ is odd, $n_{2j+1} <0$.
\\For $n_{2j+1} <0$, we have this time $n_{2i+1}=n_{\frac{t+1}{2}+j+1}'= -n_{\frac{t+1}{2}-j}'$, with $n_{\frac{t+1}{2}-j}' >0$. This positivity implies that $\frac{t+1}{2}-j-1$ is exposed at the end of a row with minimal index $\ell$, in the region $n_{\frac{t+1}{2}-j}'$.  As $\lambda$ is self-conjugate, we have:
\begin{align*}
h_{\lambda,\ell\ell}=& 2\left[(t+1)(-n_{\frac{t+1}{2}+j+1}' -1)+\frac{t+1}{2}-j \right]-1
\\=&-(2t+2)n_{2j+1}-t-1-(2j+1).
\end{align*}
As in case $1$, we derive:
\begin{equation*}\Delta_{2j+1}=-\left[(2t+2)n_{2j+1}+(2j+1)\right]-t-1. \end{equation*}

{\bf Case $3$:} $t+1$ is even, $i=2j+1$ is odd, $n_{2j+1} =0$.
\\ Here, there is no principal hook length equal to $\pm(2j+1) +t+1$ mod $2t+2$, so $\Delta_{2j+1} = (2j+1)-t-1$, and \eqref{liennideltai} is satisfied.
\smallskip

{\bf Case $4$:} $t+1$ is even, $i=2j$ is even, $n_{2j} >0$.
\\This time, by definition of $\varphi$, the involved partition is $\mu$ and $n_{2j}=n_{\frac{t+1}{2}+j+1}''$. By definition of $\phi$, the positivity of $n_{\frac{t+1}{2}+j+1}''$ implies that there is in the extended $(t+1)$-residue diagram of $\mu$ a row with minimal index $\ell$ (with $\ell \leq D(\lambda)$), such that $\frac{t+1}{2}+j$ is exposed at the end of this row, in the region $n_{\frac{t+1}{2}+j+1}''$. As $\mu$ is a doubled distinct partition, we can deduce that:

\begin{align*}
h_{\mu,\ell\ell}=& 2\left[(t+1)(n_{\frac{t+1}{2}+j+1}'' -1)+\frac{t+1}{2}+j+1 \right]-2
\\=&(2t+2)n_{2j+1}-t-1+(2j).
\end{align*}
So we have $h_{\mu,\ell\ell} \equiv 2j +t+1\mbox{~mod~}2t+2$, and this principal hook length is maximal in its congruency class.
Then, 
\begin{equation*}h_{\mu,\ell\ell}=\Delta_{2j}=\left[(2t+2)n_{2j}+(2j)\right]-t-1.\end{equation*}

{\bf Other eight cases:} they can be proved in a similar way as one of the four previous ones, by noting that when $t+1$ is odd, the roles of $\lambda$ and $\mu$ are inverted by definition of $\varphi$.
\end{proof}

\medskip
The proof of the previous theorem allows us to derive easily the following  recursive description of the inverse of $\varphi$, which is simpler than the description of $\varphi$. Fix a vector ${\bf n}=(n_1,\ldots, n_t)$ in $\mathbb{Z}^t$, then $(\lambda,\mu)=\varphi^{-1}({\bf n})$ satisfies:
 \begin{itemize}
\item if all the $n_i$'s are equal to zero, then $\lambda$ and $\mu$ are empty,
\item if a $n_i$ is equal to $1$, then the corresponding partition ($\lambda$ or $\mu$, depending on the parity of $t+1+i$) contains a principal hook of length $t+1+i$,
\item if a $n_i$ is equal to $-1$, then the corresponding partition contains a principal hook of length $t+1-i$,
\item the preimage of $(n_1,\ldots,n_i +1, \ldots, n_t)$ if $n_i>0$ (respectively $(n_1,\ldots,n_i -1, \ldots, n_t)$ if $n_i <0$) is the preimage of $(n_1,\ldots,n_i, \ldots, n_t)$ in which we add in the corresponding partition a principal hook of length $(t+1)(2n_i-1)+i$ (respectively $(t+1)(-2n_i-1)-i$) (note that it can be done in a unique way).
\end{itemize}
\begin{Remark}\label{remarquedelta} There are three immediate consequences of the previous recursive description of $\varphi^{-1}$.
\\(i) There can not be in  $\Delta$ both an integer equal to  $i+t+1$ mod $2t+2$ and an integer equal to $-i+t+1$ mod $2t+2$.
\\(ii) If $h>2t+2$ belongs to  $\Delta$, then $h-2t-2$ also belongs to $\Delta$.
\\(iii) If a finite subset of $\mathbb{N}$ satisfies the two former properties (i) and (ii) and does not contain any element equal to zero modulo $2t+2$, then it is the set $\Delta$ of a pair of $(t+1)$-cores $(\lambda,\mu)\in SC_{(t+1)} \times DD_{(t+1)}$.
\end{Remark}
By using our bijection $\varphi$, and by setting $v_i = (2t+2)n_i +i$ for $1 \leq i \leq t$, we can replace the sum in the Macdonald formula~\eqref{equaC} by a sum over pairs $(\lambda,\mu) \in SC_{(t+1)} \times DD_{(t+1)}$ (and not over vectors of integers). Therefore \eqref{equaC} takes the form (recall that $\sigma_i$ is equal to $1 $ (respectively $-1$) if $n_i \geq 0$ (respectively $n_i<0$)):
\begin{multline}\label{eqcoeff2} \displaystyle\prod\limits_{k \geq 1}(1-x^k)^{2t^2+t}  
\\= \displaystyle c_1 \sum\limits_{\lambda, \mu} x^{|\lambda| +|\mu|}\prod\limits_i ((2t+2)n_i +i)\prod\limits_{i<j}\left(((2t+2)n_i +i)^ 2-((2t+2)n_j +j)^2\right) 
\\ =\displaystyle c_1 \sum\limits_{\lambda, \mu} x^{|\lambda| +|\mu|}\prod_i \sigma_i(t+1+\Delta_i)\prod_{i<j}((t+1+\Delta_i)^2-(t+1+\Delta_j)^2).\end{multline}

\subsection{Simplification of coefficients}\label{section3.2}
The next step towards proving Theorem~\ref{theoremeintro} is a simplification of both products on the right-hand side of \eqref{eqcoeff2}, in such a way that they do not depend on the numbers $\Delta_i$ (and more generally, that they do not depend on congruency classes modulo $2t+2$). To do that, we need the following notion defined in \cite{HAN}, but only for odd integers.
 \begin{definition}A finite set of integers $A$ is a $2t+2$-\emph{compact set} if and only if the following conditions hold:
\begin{itemize}
\item[(i)] $-1,-2, \ldots, -2t-1$ belong to $A$;
\item[(ii)] for all $a \in A$ such that $a \neq -1,-2, \ldots, -2t-1$, we have $a \geq 1$ and $a \not\equiv 0~\mbox{mod~} 2t+2$;
\item[(iii)] let $b>a \geq 1$ be two integers such that $a \equiv b$ mod $2t+2$. If $b \in A$, then $a \in A$.
\end{itemize}
\end{definition}
Let A be a $2t+2$-compact set. An element $a\in A$ is $2t+2$\emph{-maximal} if for any integer $b>a$ such that $a\equiv b $ mod $2t+2$, $b \notin A$ ($i.e.$ $a$ is maximal in its congruency class modulo $2t+2$). The set of $2t+2$-maximal elements is denoted by $max_{2t+2}(A)$. It is clear by definition of compact sets that $A$ is uniquely determined by $max_{2t+2}(A)$. We can show the following lemma, whose proof is analogous to the one of \cite{HAN}, but in the even case.
\begin{lemma}\label{lemme han} For any $2t+2$-compact set A, we have:
\begin{equation}\label{eq17}
 \prod_{a \in max_{2t+2}(A)} \frac{a+2t+2}{a}=-\prod_{a\in A, a>0}\left(1-\left(\frac{2t+2}{a} \right)^2\right) .
\end{equation}
\end{lemma}

\begin{proof}
As $2t+2$ is an even integer, we have:
\begin{equation*}
\frac{1\cdot2\cdots(2t+1)}{(-2t-1)\cdots(-2)\cdot(-1)}=-1.
\end{equation*}

Then we write 
\begin{multline*}
-\prod_{a\in A, a>0}\left(1-\left(\frac{2t+2}{a} \right)^2\right) =\\ \frac{1}{(-2t-1)}\times\frac{2}{(-2t)}\times\cdots\times\frac{(2t+1)}{-1}\prod_{a\in A, a>0} \frac{(a-2t-2)}{a}\times\frac{(a+2t+2)}{a},
\end{multline*}
and the result follows by telescoping terms in the product thanks to the property (iii) of compact sets.
\end{proof}

Now the strategy is to do an induction on the number of principal hooks of the pairs $(\lambda, \mu$) appearing in \eqref{eqcoeff2}. The two following lemmas are the first step.

Let $(\lambda, \mu)$ be in $SC_{(t+1)} \times DD_{(t+1)}$ with $\lambda$ or $\mu$ non empty,  and let $\Delta$ be the set of principal hook lengths of $\lambda$ and $\mu$, from which we can define the numbers $\Delta_i$ as in Definition~\ref{deltai}. We denote by $h_{11}$ the maximal element of $\Delta$. We denote by $(\lambda', \mu') \in SC_{(t+1)} \times DD_{(t+1)}$ the pair  obtained by deleting the principal hook of length $h_{11}$. We denote by $\Delta'$ the set of principal hook lengths of $\lambda'$ and $\mu'$, and consider its associated numbers $\Delta_i'$. 

\begin{figure}[h!]
\includegraphics[scale=1.2]{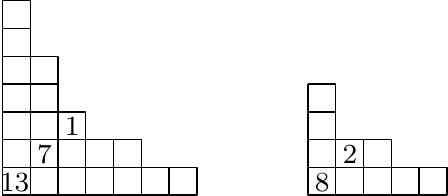}\hspace*{25pt}$\rightarrow$\hspace*{25pt}\includegraphics[scale=1.2]{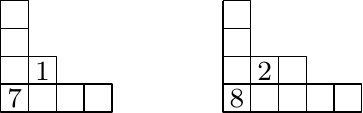}
\caption{\label{fig6}On the left, an element $(\lambda, \mu)  \in SC_{(3)} \times DD_{(3)}$. On the right, the associated $(\lambda', \mu')$.}
\end{figure}
In the example of Figure~\ref{fig6}, we have $t+1=3$, $\Delta=\{13,8,7,2,1\}$, $\Delta_1= 8$ and $\Delta_2=13$. Therefore, $\Delta'=\{8,7,2,1\}$, $\Delta_1'=8 $ and $\Delta_2'=7$.

\begin{lemma}\label{lemme 1}If $i_0$ is the (unique) integer such that $\Delta_{i_0} = h_{11}$, then we have:
\begin{multline}\label{eqlemme 1}
\prod_i \frac{\sigma_i(t+1+\Delta_i)}{\sigma_i'(t+1+\Delta_i')}\prod_{i<j}\frac{(t+1+\Delta_i)^ 2-(t+1+\Delta_j)^2}{(t+1+\Delta_i')^ 2-(t+1+\Delta_j')^2}\\= \left(1-\frac{2t+2}{h_{11}}\right)\left(1-\frac{t+1}{h_{11}}\right) \left(\frac{h_{11}}{h_{11}-2t-2}\frac{2h_{11}}{2h_{11}-2t-2}\right) \left(\frac{h_{11}+t+1}{h_{11}-t-1}\right)\\\times\prod_{j \neq i_0} \frac{(h_{11}+ \Delta_j +2t+2)(h_{11}-\Delta_j)}{(h_{11}+ \Delta_j )(h_{11}-\Delta_j -2t-2)}.
\end{multline}
\end{lemma}

\begin{proof}
First,  note that \begin{equation*}\left(1-\frac{2t+2}{h_{11}}\right)\left(1-\frac{t+1}{h_{11}}\right)=\left(\frac{h_{11}}{h_{11}-2t-2}\frac{2h_{11}}{2h_{11}-2t-2}\right)^{-1}.\end{equation*}
However, these terms are kept in this lemma in order to clarify the statement of the next one.
To prove \eqref{eqlemme 1}, we examine the consequence on the numbers $\Delta_i$ of the removal of the largest principal hook $h_{11}$. We have for all $i \neq i_0,~ \Delta_i=\Delta_i'$ and so $\sigma_i=\sigma_i'$. Indeed, the only maximum of the principal hook length congruency classes modulo $2t+2$ which can be changed by this deletion is $\Delta_{i_0}$. We can deduce that:
\begin{multline*}\prod_i \frac{\sigma_i(t+1+\Delta_i)}{\sigma_i'(t+1+\Delta_i')}\prod_{i<j}\frac{(t+1+\Delta_i)^2-(t+1+\Delta_j)^2}{(t+1+\Delta_i')^2-(t+1+\Delta_j')^2}\\= \frac{\sigma_{i_0}(t+1+\Delta_{i_0})}{\sigma_{i_0}'(t+1+\Delta_{i_0}')}\prod_{j \neq i_0}\frac{(t+1+\Delta_{i_0})^2-(t+1+\Delta_j)^2}{(t+1+\Delta_{i_0}')^2-(t+1+\Delta_j)^2}.
\end{multline*}

We want to rewrite the right-hand side of the previous product. We consider three cases depending on the value of $h_{11}=\Delta_{i_0}$.

\begin{itemize}
\item If $h_{11}>2t+2$, by Remark~\ref{remarquedelta} (i) and (ii), we know that $h_{11}-2t-2$ belongs to $\Delta$, and also to $\Delta'$. So $\Delta_{i_0}'=\Delta_{i_0}-2t-2$. In this case, we have $\sigma_{i_0}=\sigma_{i_0}'$ and \begin{equation*}\frac{\sigma_{i_0}(t+1+\Delta_{i_0})}{\sigma_{i_0}'(t+1+\Delta_{i_0}')}=\frac{h_{11}+t+1}{h_{11}-t-1}.\end{equation*} 
Finally, we can immediately rewrite the product over $j \neq i_0$   to obtain \eqref{eqlemme 1}.
\item If $t+1<h_{11}<2t+2$, by the definition of $\Delta_{i_0}$ and Remark~\ref{remarquedelta} (i), we know that $h_{11}=t+1+i_0$ and $\sigma_{i_0}=1$. In this case, as we delete the only hook equal to $\pm i +t+1$ mod $2t+2$, we have $\Delta_{i_0}'= i_0 -t-1= h_{11} -2t-2$, and $\sigma_{i_0}'=1$. Therefore \begin{equation*}\frac{\sigma_{i_0}(t+1+\Delta_{i_0})}{\sigma_{i_0}'(t+1+\Delta_{i_0}')}=\frac{h_{11}+t+1}{h_{11}-t-1}, \end{equation*}
and moreover the product can be directly rewritten as in \eqref{eqlemme 1}.
\item If $0<h_{11}<t+1$, by the definition of $\Delta_{i_0}$ and Remark~\ref{remarquedelta} (i), we know that $h_{11}=t+1-i_0$ and $\sigma_{i_0}=-1$. In this case, we have $\Delta_{i_0}'= i_0 -t-1= -h_{11}$ and $\sigma_{i_0}'=1$. We have:
\begin{equation*}
\frac{\sigma_{i_0}(t+1+\Delta_{i_0})}{\sigma_{i_0}'(t+1+\Delta_{i_0}')}= - \frac{t+1+h_{11}}{t+1-h_{11}}=\frac{h_{11}+t+1}{h_{11}-t-1}.
\end{equation*}
The form \eqref{eqlemme 1} for the product over $j \neq i_0$ again comes straightforwardly after noticing that 

\begin{equation*}(t+1+ \Delta_{i_0}')^2 =(t+1 -h_{11})^2= (h_{11}-t-1)^2. \end{equation*}
\end{itemize}
\end{proof}

\begin{lemma}\label{lemme 2}
With the same notations as above, we define the set $E$ as:\begin{equation}\label{defE}\left(\bigcup \limits _{j \neq i_0} \{h_{11}+ \Delta_j, h_{11}-\Delta_j -2t-2\}\right) \cup\{h_{11}-t-1, h_{11}-2t-2, 2h_{11}-2t-2\}. \end{equation} Then there exists a unique $2t+2$-compact set $H$, such that $E= max_{2t+2}(H)$. Moreover, its subset $H_{>0}$ of positive elements is independent of $t$ and made of elements of the form  $h_{11} + \tau_m m$, where $1\leq m\leq h_{11}-1$, and $\tau_m$ is equal to $1$ if $m$ is a principal hook length ($i.e.~m \in \Delta$) and to $-1$ otherwise. 
\end{lemma}

\begin{proof}
To show that the set $E$ is the $max_{2t+2}(H)$ of a set $H$, it is sufficient by definition to verify that $E$ contains exactly $2t+1$ elements, that all of them are distinct modulo $2t+2$, and that none of them is equal to $0$ modulo $2t+2$. All of these properties result from the definition of the numbers $\Delta_i$ (see Definition~\ref{deltai}). Moreover, unicity is immediate by definition of $2t+2$-compact sets.

It remains to show that the set $H_{>0}$ and the set $H'$ made of elements of the form  $h_{11} +\tau_m m$, where $1\leq m\leq h_{11}-1$, are the same. We start by proving the inclusion $H_{>0} \subset H'$. To this aim, we notice first that all the positive elements of $E$ are also elements of $H_{>0}$. Moreover, we can describe all the elements of $H_{>0}$ from the elements of $E$; more precisely, from each element $x$ of $E$, we will describe all the elements of $H_{>0}$ which are in the same congruency class modulo $2t+2$ as $x$. It will remain to show that these elements are necessarily also in $H'$.

Let $i$ be the rest in the euclidean division of $h_{11}$ by $2t+2$, then $h_{11}=2k(t+1)+i$ and two cases can occur: either $i \in \{1,\ldots, t\}$ or $i \in \{t+2,\ldots, 2t+1\}$. We only treat here the first case, the second can be done in the same way.

The element $x:=h_{11}-(t+1)$ in $E$  is strictly positive if and only if  $k>0$, which implies that the integers belonging to $H_{>0}$ which are in the same congruency class modulo $2t+2$ as  $x$  are exactly the integers $h_{11}-(t+1),  h_{11}-3(t+1), \ldots, h_{11}-(2k-1)(t+1)$. Each integer $m \in \{(t+1), 3(t+1),\ldots,  (2k-1)(t+1)\}$  is not a principal hook length of $\lambda$ or $\mu$ and is smaller than $h_{11}$,  and so we have $\tau_m=-1$. Therefore $h_{11}-m$ belongs to $H'$.

The element $x:=h_{11}-2(t+1)$ in $E$  is strictly positive if and only if  $k>0$, which implies that the integers belonging to $H_{>0}$ which are in the same congruency class modulo $2t+2$ as  $x$  are exactly the integers $h_{11}-2(t+1),  h_{11}-4(t+1), \ldots, h_{11}-2k(t+1)$. Each integer $m \in  \{2(t+1), 4(t+1),\ldots,  2k(t+1)\}$  is not a principal hook length of $\lambda$ or $\mu$ and is smaller than $h_{11}$, so we have $\tau_m=-1$. Therefore $h_{11}- m$ belongs to $H'$.

The element $x:=2h_{11}-2t-2$ in $E$ is strictly positive if and only if  $k>0$, which implies that the  integers belonging to $H_{>0}$ which are in the same congruency class modulo $2t+2$ as  $x$  are exactly the integers $2h_{11}-2(t+1), 2h_{11}-4(t+1), \ldots, 2h_{11}-4k(t+1)$. These elements all belong to $H'$ for one of the two following reasons.  First, each integer $m \in \{(2k-2)(t+1)+i, (2k-4)(t+1)+i, \ldots, i\}$  is a principal hook length smaller than $h_{11}$ according to Remark~\ref{remarquedelta}, so  we have $\tau_m=1$ and therefore $2h_{11}-2(t+1), \ldots,2h_{11}-2k(t+1)$ all belong to $H'$. Second, each integer $m \in \{2k(t+1)-i, (2k-2)(t+1)-i, \ldots, 2(t+1)-i\}$ is not a principal hook length according to Remark~\ref{remarquedelta} and is smaller than $h_{11}$, so we have $\tau_m=-1$ and therefore $2h_{11}-4k(t+1), \ldots,2h_{11}-(2k+2)(t+1)$ all belong to $H'$.

For the other elements $x$ of $E$, we must again consider two cases. If we make the euclidean division of $\Delta_{t+1-j}$ by $2(t+1)$, we have either $\Delta_{t+1-j}=2\ell(t+1)+j$ or $\Delta_{t+1-j}=2\ell(t+1)-j$ (by definition of $\Delta_{t+1-j}$). Again, we only treat here the first case, as the second can be done in the same way.

The element $x:=h_{11}+ \Delta_{t+1-j}= 2(k+l)(t+1)+i+j$ in $E$ is strictly positive which implies that the integers belonging to $H_{>0}$ which are in the same congruency class modulo $2t+2$ as  $x$  are exactly the integers $2(k+\ell)(t+1)+i+j, 2(k+\ell-1)(t+1)+i+j, \ldots, i+j$. The element $y:=h_{11} - \Delta_{t+1-j}-2(t+1)= 2(k-\ell-1)(t+1)+i+j$  in $E$ is strictly positive if and only if $k>\ell$ which implies that the integers belonging to $H_{>0}$ which are in the same congruency class modulo $2t+2$ as  $y$  are exactly the integers $2(k-\ell-1)(t+1)+i-j, \ldots, 2(t+1)+i-j, i-j$ (the integer $i-j$ must be considered only if it is positive). All these elements are also in $H'$ for one of the two following reasons. First, each integer $m \in \{j, 2(t+1)+j,  \ldots, 2\ell(t+1) +j\}$ is a principal hook length of $\lambda$ or $\mu$ (according to Remark~\ref{remarquedelta}), so we have $\tau_m =1$. Therefore $2k(t+1)+i+j, \ldots,2(k+\ell)(t+1)+i+j$ all belong to $H'$.  Second, each integer $m \in \{2(t+1)-j,\ldots, 2k(t+1)-j\}$ or  $m \in \{2(\ell+1)(t+1)+j, \ldots, 2k(t+1) +j\}$ is not a principal hook length (by virtue of the same Remark or by definition of $\Delta_{t+1-j}$), therefore $\tau_m = -1$. So the integers $2(k-1)(t+1)+i+j, \ldots, i+j$ and $2(k-\ell-1)(t+1)+i-j, \ldots, 2(t+1)+i-j, i-j$ are all elements of $H'$.

Finally we must verify that all the elements of $H'$ also belong to $H_{>0}$, which can be done similarly, as each of the previous steps can be performed in the reverse sense.
\end{proof}

Now, we are able to derive the following lemma.
\begin{lemma}\label{lemme-recap}
If $(\lambda, \mu)$ is in $SC_{(t+1)} \times DD_{(t+1)}$ and $ (n_1,\ldots, n_t):= \varphi(\lambda, \mu)$, then the following equality holds:
\begin{eqnarray}\label{lemme4g}
\prod_i ((2t+2)n_i +i)\prod_{i<j}\left(((2t+2)n_i +i)^ 2-((2t+2)n_j +j)^2\right)\hspace*{2cm}\\= \frac{\delta_\lambda \delta_\mu}{c_1}\prod_{h \in \Delta} \left(1-\frac{2t+2}{h}\right)\left(1-\frac{t+1}{h}\right) \prod_{j=1}^{h-1}\left( 1-\left(\frac{2t+2}{h +\tau_j j} \right)^2\right)\label{lemme4d},
\end{eqnarray}
where $\delta_\lambda$ and $\delta_\mu$ are defined in Section~\ref{defs}.
\end{lemma}
\begin{proof}
The strategy is to do an induction on the number of principal hooks of the pair $(\lambda,\mu)$, by deleting at each step the largest principal hook $h_{11}$ in $\Delta$. To this aim, we denote by $P$ the term in \eqref{lemme4g} and, by using $\varphi$, we transform $P$ into products involving the numbers $\Delta_i$, as we did when we established \eqref{eqcoeff2}:
\begin{equation*}
P= \prod_i \sigma_i(t+1+\Delta_i)\prod_{i<j}((t+1+\Delta_i)^2-(t+1+\Delta_j)^2).
\end{equation*}
By using notations before Lemma~\ref{lemme 1}, we define in the same way
\begin{equation*}
P':= \prod_i \sigma_i'(t+1+\Delta_i')\prod_{i<j}((t+1+\Delta_i')^2-(t+1+\Delta_j')^2).
\end{equation*}
Lemma~\ref{lemme 1} gives:
\begin{multline}\label{PP'}
P= \left(\frac{h_{11}}{h_{11}-2t-2}\frac{2h_{11}}{2h_{11}-2t-2}\frac{h_{11}+t+1}{h_{11}-t-1}\right)\prod_{j \neq i_0} \frac{(h_{11}+ \Delta_j +2t+2)(h_{11}-\Delta_j)}{(h_{11}+ \Delta_j )(h_{11}-\Delta_j -2t-2)}
\\\times  \left(1-\frac{2t+2}{h_{11}}\right)\left(1-\frac{t+1}{h_{11}}\right) \times P'.
\end{multline} 

Then Lemma~\ref{lemme 2} insures that the set $E$ defined as \eqref{defE} is the $max_{2t+2}(H)$ of a unique $2t+2$-compact set $H$, with a subset of positive elements $H_{>0}=\{h_{11}- \tau_m m, 1\leq m \leq h_{11}-1\}$. So we can apply Lemma~\ref{lemme han} to show that
\begin{multline*}
\left(\frac{h_{11}}{h_{11}-2t-2}\frac{2h_{11}}{2h_{11}-2t-2}\frac{h_{11}+t+1}{h_{11}-t-1}\right)\prod_{j \neq i_0} \frac{(h_{11}+ \Delta_j +2t+2)(h_{11}-\Delta_j)}{(h_{11}+ \Delta_j )(h_{11}-\Delta_j -2t-2)}
\\=-\prod_{j=1}^{h_{11}-1}\left( 1-\left(\frac{2t+2}{h_{11} +\tau_j j} \right)^2\right).\hspace*{5cm}
\end{multline*}
So \eqref{PP'} becomes:
\begin{equation*}
P=-\left(1-\frac{2t+2}{h_{11}}\right)\left(1-\frac{t+1}{h_{11}}\right)\prod_{j=1}^{h_{11}-1}\left( 1-\left(\frac{2t+2}{h_{11} +\tau_j j} \right)^2\right) \times P'.
\end{equation*}
  Next we do an induction on the cardinality of $\Delta$ by deleting in $\lambda$ or $\mu$ the largest element of $\Delta$. There are $D(\lambda)+D(\mu) $ steps in the induction, each of which giving rise to a minus sign. This explains the term $\delta_\lambda \delta_\mu$.  The base case corresponds to empty partitions $\lambda$ and $\mu$. In this case $\Delta_i =i-t-1$, $1\leq i\leq t$, therefore
\begin{eqnarray*}
\prod_i \sigma_i(t+1+\Delta_i)\prod_{i<j}\left((t+1+\Delta_i)^ 2-(t+1+\Delta_i)^2\right)=\prod_i i \prod_{i<j}\left( i^2-j^2\right)= \frac{1}{c_1}, 
\end{eqnarray*}
where we recall that $c_1$ is defined in \eqref{equaC}.
\end{proof}

\subsection{Proof of Theorem~\ref{theoremeintro}}\label{section3.3}

We can actually prove the following result, which will be seen to be equivalent to Theorem~\ref{theoremeintro}.
\begin{Theorem}\label{thmprincipal}
The following identity holds for any complex number t:
\begin{multline}\label{eqprincip}
\prod\limits_{n \geq 1}(1-x^n)^{2t^2+t} = \sum_{\lambda,\mu} \delta_\lambda \delta_\mu x^{|\lambda|+ |\mu|}\\\times \prod\limits_{h \in \Delta} \left(1-\frac{2t+2}{h}\right)\left(1-\frac{t+1}{h}\right) \prod\limits_{j=1}^{h-1} \left(1-\left(\frac{2t+2}{h + \tau_j j} \right)^2 \right),
\end{multline}
where the sum ranges over pairs $(\lambda, \mu)$ of partitions,  $\lambda$ being self-conjugate and $\mu$ being doubled distinct.
\end{Theorem}
\begin{proof}
Thanks to the Macdonald formula~\eqref{equaC} and Lemma~\ref{lemme-recap}, equation~\eqref{eqprincip} holds if the sum on the right-hand side is over pairs $(\lambda, \mu) \in SC_{(t+1)} \times DD_{(t+1)}$ and if $t$ is a positive integer. We will show that the product 
\begin{equation*} Q :=\prod_{h \in \Delta}\left( 1-\frac{2t+2}{h}\right)\left(1-\frac{t+1}{h}\right) \prod_{j=1}^{h-1} \left(1-\left(\frac{2t+2}{h +\tau_j j} \right)^2 \right)
\end{equation*} vanishes if the pair $(\lambda, \mu)$ is not a pair of $t+1$-cores. Indeed, set $(\lambda, \mu) \in SC \times DD$, and let $\Delta$ be the set of principal hook lengths of $\lambda$ and $\mu$. We show that $Q$ vanishes unless $\Delta$ satisfies the three hypotheses of (iii) in Remark~\ref{remarquedelta}. Assume $Q \neq 0$.

First, let $h>2t+2$  be an element of $\Delta$.  If $j:=h-2t-2$  was not a principal hook length, then the term corresponding to $j$ in the second product of $Q$ would vanish by definition of $\tau_j$. So (ii) is satisfied.

Second, let $k, k',i$ be nonnegative integers such that $1\leq i \leq t$. If $(2k+1)(t+1)+i$ and $(2k'+1)(t+1)-i$ both belong to $\Delta$, then by induction and according to the previous case, $t+1+i$ and $t+1-i$ belong to $\Delta$. But then, the term $1-\left(\frac{2t+2}{(t+1+i)+(t+1-i)}\right)^2$ vanishes, which contradicts $Q \neq 0$. So $(2k+1)(t+1)+i$ and $(2k'+1)(t+1)-i$ can not be both principal hook lengths. So (i) is satisfied.

Third, if $\Delta$ contains multiples of $t+1$, we denote by $h$ the smallest such principal hook length. If  $h= t+1$ or $h=2t+2$, then the first term of the product $Q$ would vanish. Otherwise, $h-2t-2$ does not belong to $\Delta$ by minimality, and so the term corresponding to $j=h-2t-2$ in the second product of $Q$  would vanish. Therefore $\Delta$ can not contain multiples of $t+1$.

 According to Remark~\ref{remarquedelta} (iii), if $Q \neq 0$, then  $(\lambda, \mu)$ is a pair of $t+1$-cores. So formula \eqref{eqprincip} remains true for any positive integer $t$ if the sum ranges over $SC \times DD$.
 
  To conclude, we give a polynomiality argument which generalizes \eqref{eqprincip} to all complex numbers $t$. To this aim, we can use the following formula:
\begin{equation}\label{eqex}
\prod_{k \geq 1} \frac{1}{1-x^k} = \mbox{exp}\left(\sum_{k \geq 1} \frac{x^k}{k(1-x^k)}\right),
\end{equation}
in order to rewrite the left-hand side of~\eqref{eqprincip} in the form
\begin{equation}\label{cotegauche}
\mbox{exp}\left(-(2t ^2+t)\sum_{k \geq 1} \frac{x^k}{k(1-x^k)}\right).
\end{equation}

Let $m$ be a nonnegative integer. The coefficient $C_m(t)$ of $X^m$ on the left-hand side of~\eqref{eqprincip} is a polynomial in $t$, according to~\eqref{cotegauche}, as is the coefficient $D_m(t)$ of $X^m$ on the right-hand side. Formula~\eqref{eqprincip} is true for all integers $t \geq 2$, it is therefore still true for any complex number $t$.
\end{proof}

Let $(\lambda, \mu)$ be in $SC \times DD$, with set of principal hook lengths $\Delta$. We denote by $2\Delta$ the set of elements of $\Delta$ multiplied by $2$. Note that we can bijectively associate to $(\lambda, \mu)$ a partition $\nu \in DD$ with set of principal hook lengths $2\Delta$, as illustrated below.
\begin{figure}[h!]
\includegraphics[scale=1.2]{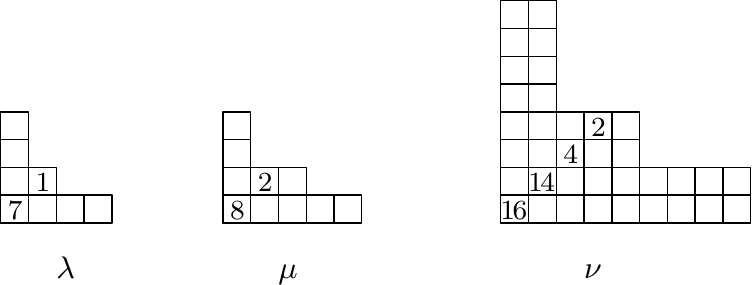}
 \caption{\label{fig8}Example of construction of $\nu$.}
\end{figure}
\begin{Theorem}\label{bijcouple}
The partition $\nu$ satisfies $|\lambda|+|\mu| = |\nu|/2$, $\delta_\lambda \delta_\mu = \delta_\nu$, and
\begin{multline}\prod\limits_{h \in \Delta} \left(1-\frac{2t+2}{h}\right)\left(1-\frac{t+1}{h}\right) \prod\limits_{j=1}^{h-1}\left(  1-\left(\frac{2t+2}{h +\tau_j j} \right)^2 \right)\\= \prod_{h \in \nu} \left(1-\frac{2t+2}{h \, \varepsilon_h}\right),\label{eqbij}
\end{multline}
 where $\varepsilon_h$ is equal to $-1$ if $h$ is the hook length of a box  strictly above the principal diagonal, and to $1$ otherwise.
\end{Theorem}

\begin{proof}

The two first properties are clear by definition of $\nu$ and $\delta_\nu$, the difficult point is \eqref{eqbij}. We prove it by induction on the number of elements of $\Delta$. The result is trivial if $\lambda$ and $\mu$ are both empty. If $\lambda$ or $\mu$ is not empty, we denote by $\Delta^1 > \Delta^2> \cdots>\Delta^\ell$ the elements of $\Delta$, and write $\nu =(\nu_1, \nu_2, \ldots, \nu_k$). We will show the following identity:
\begin{equation}\label{eq5.1}
\left(1-\frac{2t+2}{\Delta^1}\right)\left(1-\frac{t+1}{\Delta^1}\right) \prod\limits_{j=1}^{\Delta^1-1}\left(  1-\left(\frac{2t+2}{\Delta^1 +\tau_j j} \right)^2 \right)= \prod\limits_{h} \left(1-\frac{2t+2}{h \, \varepsilon_h}\right),
\end{equation}
where the product on the right ranges over the boxes of the largest principal hook of $\nu$. This is enough to establish the induction.

We denote by $h_{\nu, i,j}$ the hook length of the box $(i,j) $ in $\nu$.
According to the definition of $\nu$, we have $ h_{\nu,1,1}=2\Delta^1$ and $D(\nu)=\ell=\# \Delta$. As $\nu$ is a doubled distinct partition, a direct computation of the hook length gives $h_{\nu,1,\ell+1}= \Delta^1$. Therefore \begin{equation} \label{eq5.2} \left(1-\frac{2t+2}{\Delta^1}\right)\left(1-\frac{t+1}{\Delta^1}\right)= \left(1-\frac{2t+2}{h_{\nu,1,1} \, \varepsilon_{h_{\nu,1,1}}}\right)\left(1-\frac{2t+2}{h_{\nu,1,\ell+1} \, \varepsilon_{h_{\nu,1,\ell+1}}}\right).\end{equation}

Let $i$ be in $\{2,\ldots, \ell\}$. As $\nu$ is a doubled distinct partition, we have $h_{\nu,i,1}= h_{\nu,1,i}$ by direct computation. By construction of $\nu$, we have $h_{\nu,i,1}= \nu_i-i+\nu_1^*=\Delta^i+\Delta^1$. As $\varepsilon_{h_{\nu,i,1}}= -\varepsilon_{h_{\nu,1,i}}$, we get:
\begin{equation} \label{eq5.3}  1-\left(\frac{2t+2}{\Delta^1 +\Delta^i} \right)^2=\left(1-\frac{2t+2}{h_{\nu,i,1} \, \varepsilon_{h_{\nu,i,1}}}\right)\left(1-\frac{2t+2}{h_{\nu,1,i} \, \varepsilon_{h_{\nu,1,i}}}\right).\end{equation}

Let $i$ be in $\{\ell+1,\ldots, \Delta^1\}$, write $i=\ell+i'$ with $i'$ in $\{1,\ldots, \Delta^1-\ell\}$, and denote by $\nabla_{\Delta^1-\ell}>\ldots>\nabla_2>\nabla_1$ the integers between $1$ and $\Delta^1$ which do not belong to $\Delta$. For example, in Figure~\ref{fig8}, we have $(\nabla_{4},\nabla_3,\nabla_2,\nabla_1)=(6,5,4,3)$. As $\nu$ is a doubled distinct partition, we have $h_{\nu,i,1}= h_{\nu,1,i+1}$. By construction of $\nu$, we have $\nu_i-i=-\nabla_{i'}$, and so $h_{\nu,i,1}= \nu_i-i+\nu_1^*=\Delta^1-\nabla_{i'}$. As $\varepsilon_{h_{\nu,i,1}}= -\varepsilon_{h_{\nu,1,i+1}}$, we derive:
\begin{equation} \label{eq5.4}   1-\left(\frac{2t+2}{\Delta^1 -\nabla_{i'}} \right)^2=\left(1-\frac{2t+2}{h_{\nu,i,1} \, \varepsilon_{h_{\nu,i,1}}}\right)\left(1-\frac{2t+2}{h_{\nu,1,i+1} \, \varepsilon_{h_{\nu,1,i+1}}}\right).\end{equation}

The product of \eqref{eq5.2}, \eqref{eq5.3}  and \eqref{eq5.4} over all $i \in \{1,\ldots, \Delta^1\}$ gives \eqref{eq5.1}, as we have:
\begin{equation*}
\prod\limits_{j=1}^{\Delta^1-1}\left(  1-\left(\frac{2t+2}{\Delta^1 +\tau_j j} \right)^2 \right)= \prod_{i =2}^\ell \left(1-\left(\frac{2t+2}{\Delta^1 +\Delta^i} \right)^2 \right) \prod_{i'=1}^{\Delta^1-\ell} \left( 1-\left(\frac{2t+2}{\Delta^1 -\nabla_{i'}} \right)^2 \right).
\end{equation*}
\end{proof}

 Theorem~\ref{theoremeintro} straightforwardly follows from Theorems~\ref{thmprincipal} and \ref{bijcouple}.

\subsection{Some applications and adaptation to types $\widetilde{B}$ and $\widetilde{BC}$.}\label{section3.4}\label{section3}
We give here some direct applications of Theorem~\ref{theoremeintro}.  First, taking $t=-1$ in \eqref{eqtheoremeintro} yields the following famous expansion (see for example \cite[p. 69]{EC1}), where the sum ranges over partitions with distinct parts:
\begin{equation} \label{partpartdistinct}
\prod_{n \geq 1}(1-x^n) = \sum_\lambda (-1)^{\#\{\scriptsize{parts~of~}\lambda\}} x ^{|\lambda|}.
\end{equation}

Next, recall the classical hook formula (see for instance~\cite[Corollary 7.12.6]{EC})
\begin{equation*}
\sum_{\stackrel{\lambda \in \mathcal{P}}{|\lambda|=n}}\prod_{h \in \mathcal{H}(\lambda)} \frac{1}{h^2}=\frac{1}{n!},
\end{equation*}
which can be proved through the Robinson--Schensted--Knuth correspondence and which is famous in representation theory for the Coxeter group of type $A$ (or symmetric group). From this formula, and by extracting coefficients while using \eqref{eqtheoremeintro} and \eqref{nekrasov}, we can derive the following result, which can be seen as a \emph{symplectic hook formula}, as is explained by the denominator on the right-hand side, which is the cardinal of the Coxeter group of type $B_n$.
\begin{corollary} For any positive integer $n$, we have:
\begin{equation}
\sum_{\stackrel{\lambda \in DD}{|\lambda|= 2n}} \prod_{h \in \mathcal{H}(\lambda)}\frac{1}{h} = \frac{1}{2^n n!}.
\end{equation}
\end{corollary}
We do not detail yet the proof of this formula, as it is also a direct consequence of Corollary~\ref{cor9} below.

Finally, we can prove the following result,
which establishes a surprising link between the Macdonald formulas in types $\widetilde{C}$, $\widetilde{B}$, and $\widetilde{BC}$.

\begin{Theorem}\label{equi}
The following families of formulas are all generalized by Theorem~\ref{theoremeintro}:
\begin{itemize}\itemsep-2.5pt
\item[(i)]the Macdonald formula \eqref{equaC} in type $\widetilde{C}_t$ for any integer $t \geq 2$; \medskip
\item[(ii)]the Macdonald formula in type $\widetilde{B}_t$ for any integer $t \geq 3$:
\begin{equation*}\label{equaB}
\eta(x)^{2t^2+t} = c_1 \sum_{{\bf v}} x^{\|{\bf v}\|^ 2/8(2t-1)} \prod_i v_i \prod_{i<j}(v_i^2-v_j^2) , 
\end{equation*}
where the sum ranges over $t$-tuples ${\bf v} :=(v_1,\ldots, v_t) \in \mathbb{Z}^t$ such that $v_i \equiv 2i-1 \mbox{~mod~} 4t-2$ and $v_1+\cdots+v_t= t^2 \mbox{~mod~} 8t-4$;\medskip
\item[(iii)]the Macdonald formula in type $\widetilde{BC}_t$ for any integer $t\geq 1$:\end{itemize}
\begin{equation*}\label{equaBC}
\eta(x)^{2t^2-t} =c_2 \sum_{{\bf v}} x^{\|{\bf v}\|^2 /8(2t+1)} (-1)^{(v_1+\cdots+v_t-t)/2}\prod_{i<j}(v_i^2-v_j^2),
\end{equation*}
\hspace*{0.9cm}with $\displaystyle c_2:= \frac{(-1)^{(t-1)/2}t!}{1! 2! \cdots (t-1)!},$ and where the sum ranges over $t$-tuples ${\bf v} :=\hspace*{0.9cm}(v_1,\ldots, v_t) \in \mathbb{Z}^t$ such that $v_i \equiv 2i-1 \mbox{~mod~} 4t+2$.

\end{Theorem}
\begin{proof} As the methods are similar to the proof of Theorem~\ref{theoremeintro}, we only highlight the ideas. By substituting $u:=-t-1/2$  in \eqref{eqtheoremeintro}, and considering the positive integral values of $u$, we first prove that the product on the right-hand side vanishes for all partitions $\lambda$, except for those that do not contain a hook length equal to $2u-1$ for boxes strictly above the diagonal. By using some lemmas analogous to Lemmas~\ref{lemme han}--\ref{lemme-recap} (in the reverse sense) and a bijection analogous to $\varphi$, we manage to derive the Macdonald formula in type $\widetilde{B}_u$ for any integer $u \geq 3$. 
The same reasoning applies for type $\widetilde{BC}_\ell$ by doing the substitution $\ell:=t-1/2$ for integers $\ell \geq 1$. The partitions $\lambda$ that occur here are $2\ell+1$-cores.
\end{proof}

\subsection{Refinement of a result due to Kostant}\label{section3.5}

Let us write 
\begin{equation*}
\prod_{n \geq 1} (1-x^ n)^ s= \sum_{k \geq 0} f_k(s) x^ k.
\end{equation*}
Kostant proved through considerations on Lie algebras the following result \cite[Theorem~4.28]{KOS}.
\begin{Theorem}[\cite{KOS}]\label{thmkos}
Let $k$ and $m$ be two positive integers such that $m \geq \max (k,4)$. Then $f_k(m ^2-1) \neq 0$.
\end{Theorem}
Notice that the original statement $m>1$ in Kostant's Theorem should be replaced by $m\geq 4$, as noticed by Han. This theorem was refined by Han in \cite[Theorem~1.6]{HAN} in the following way.
\begin{Theorem}[\cite{HAN}]\label{kostanthan}
Let $k$ be a positive integer and $s$ be a real number such that $s \geq k^2-1$. Then $(-1)^kf_k(s)>0$
\end{Theorem}
In the same vein as did Han, we can refine differently Theorem~\ref{thmkos}.
\begin{Theorem}\label{kostantpet}
Let $k$ be a positive integer and $s$ be a real number such that $s>k-1$. Then $(-1)^kf_k(2s^2+s)>0$.
\end{Theorem}

\begin{proof}
By Theorem~\ref{theoremeintro}, we can write
\begin{equation*}
f_k(2s^2+s)= \sum_{\lambda \in DD, |\lambda|=2k} W(\lambda),
\end{equation*}
where $\displaystyle W(\lambda):=\delta_\lambda \prod_{h \in \mathcal{H}(\lambda)}\left(1-\frac{2s+2}{h\,\varepsilon_h}\right)$. 

By denoting $\mathcal{H}(\lambda)_{\leq}$ (respectively $\mathcal{H}(\lambda)_{>}$) the multi-set of hook lengths of boxes below (respectively strictly above) the principal diagonal of $\lambda$, we have
\begin{eqnarray*}
\displaystyle W(\lambda)&=&\displaystyle\delta_\lambda \prod_{h \in \mathcal{H}(\lambda)}\left(\frac{h-\varepsilon_h(2s+2)}{h}\right)
\\&=&\displaystyle\delta_\lambda \prod_{h \in \mathcal{H}(\lambda)_\leq}\left(\frac{h-(2s+2)}{h}\right)\prod_{h \in \mathcal{H}(\lambda)_>}\left(\frac{h+(2s+2)}{h}\right)
\\&=&\displaystyle \delta_\lambda (-1) ^{\# \mathcal{H}(\lambda)_\leq}\prod_{h \in \mathcal{H}(\lambda)_\leq}\left(\frac{2s+2-h}{h}\right)\prod_{h \in \mathcal{H}(\lambda)_>}\left(\frac{h+2s+2}{h}\right)
\\&=&\displaystyle (-1)^ k\prod_{h \in \mathcal{H}(\lambda)_\leq}\left(\frac{2s+2-h}{h}\right)\prod_{h \in \mathcal{H}(\lambda)_>}\left(\frac{h+2s+2}{h}\right),
\end{eqnarray*}
where the last equality is a consequence of $(-1)^{\#\mathcal{H}(\lambda)_\leq}=(-1)^k\delta_\lambda $. As $|\lambda|=2k$, the condition $2s+2>2k$ implies that both products are positive. Hence $(-1)^kW(\lambda)>0$ and the result follows.
\end{proof}

\section{A generalization through Littlewood decomposition}\label{section4}
In this section we prove a generalization of Theorem~\ref{theoremeintro} with two extra parameters, by using the canonical correspondence between partitions and bi-infinite binary words beginning with infinitely many $0$'s and ending with infinitely many $1$'s, together with new properties of the so-called Littlewood decomposition, which we first recall in the next subsection.

\subsection{Littlewood decomposition}\label{section4.1}

We follow Han \cite{HAN} here. The Littlewood decomposition is a classical bijection which maps each partition to its $t$-core and $t$-quotient (see for example \cite[p. 468]{EC}).  Let $\mathcal{W}$ be the set of bi-infinite binary sequences beginning with infinitely many $0$'s and ending with infinitely many $1$'s. Each element $w$ of $\mathcal{W}$ can be represented by a sequence $(b_i)_i= \cdots b_{-2}b_{-1}b_0b_1b_2 \cdots$, but the representation is not unique. Indeed, for any fixed integer $k$ the sequence $(b_{i+k})_i$ represents $w$. The \emph{canonical representation} of $w$ is the unique sequence $(c_i)_i=\cdots c_{-2}c_{-1}c_0c_1c_2 \cdots$ such that:
\begin{equation*}
\#\{i \leq -1, c_i=1\}= \#\{i \geq 0, c_i=0\}.
\end{equation*}

We put a dot symbol ``." between the letter $c_{-1}$ and $c_0$ in the bi-infinite sequence $(c_i)_i$ when it is the canonical representation.
There is a natural one-to-one correspondence between the set of partitions $\mathcal{P}$ and $\mathcal{W}$. Let $\lambda$ be a partition. We encode each horizontal edge of the Ferrers diagram of $\lambda$ by $1$ and each vertical edge by $0$. Reading these $(0,1)$-encodings from top to bottom and left to right yields a binary word $u$. By adding infinitely many $0$'s to the left and infinitely many $1$'s to the right of $u$, we get an element $w=\cdots000u111\cdots \in \mathcal{W}$. The map 
\begin{equation}\label{defpsi}\psi : \lambda \mapsto w \end{equation} 
is a one-to-one correspondence between $\mathcal{P}$ and $\mathcal{W}$. The canonical representation of $\psi(\lambda)$ will be denoted by $C_\lambda$. 
\begin{figure}[h!]
\includegraphics[scale=1.2]{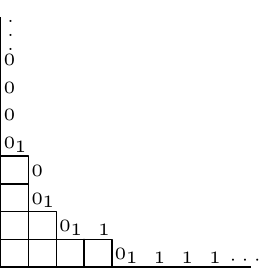}\hspace*{50pt}\includegraphics[scale=1.2]{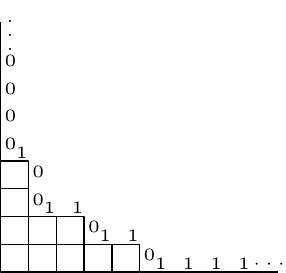}
\caption{\label{fig9} A self-conjugate and a doubled distinct partition, and their respective $(0,1)$-encodings.}
\end{figure}

For example (see Figure~\ref{fig9}), for $\lambda=(4,2,1,1)$, we have $u=10010110$, so that \\$w=\cdots00010010110111\cdots$ and $C_\lambda=\cdots0001001.0110111\cdots$.

Notice that the symbol ``." in the canonical representation of $\lambda$ corresponds in the Ferrers diagram to the corner of its Durfee square. The size of the Durfee square is the number of $1$'s before the symbol ``." in the canonical representation. The following definition is useful to interpret self-conjugate and doubled distinct partitions in terms of words.

\begin{definition}\label{definitionf}
Let $v$ be a finite binary word. We define $f(v)$ as the reverse word of $v$ in which we exchange the letters $0$ and $1$.
\end{definition}

For example, if $v=1001010$, then $f(v)=1010110$.

 Notice that $\lambda $ is a doubled distinct partition if and only if its canonical representation is of the form $\cdots00v.1f(v)11\cdots$, where $v$ is a finite word. Notice finally that $\lambda$ is self-conjugate if and only if its canonical representation is of the form $\cdots00v.f(v)11\cdots$.

Let $t$ be a positive integer. The aforementioned \emph{Littlewood decomposition} is a classical bijection $\Omega$, which maps a partition $\lambda$ to ($\tilde{\lambda}, \lambda^0, \lambda^1,\ldots, \lambda^{t-1}$) such that:
\begin{itemize}
\item[(i)] $\tilde{\lambda}$ is the $t$-core of $\lambda$ and $\lambda^0, \lambda^1,\ldots, \lambda^{t-1}$ are partitions;
\item[(ii)]$|\lambda|= |\tilde{\lambda}|+t(|\lambda^0|+ |\lambda^1|+\cdots+ |\lambda^{t-1}|)$;
\item[(iii)]$\{h/t, h \in \mathcal{H}_t(\lambda)\}=\mathcal{H}(\lambda^0)\cup\mathcal{H}(\lambda^1)\cup \cdots \cup \mathcal{H}(\lambda^{t-1})$, where this equality should be understood in terms of multi-sets. Moreover, there is a canonical one-to-one correspondence between the boxes of $\lambda$ with hook lengths which are multiples of $t$ and the boxes of $\lambda^0, \lambda^1,\ldots,\lambda^{t-1}$.
\end{itemize}

\begin{figure}[h!]
\includegraphics[scale=1.2]{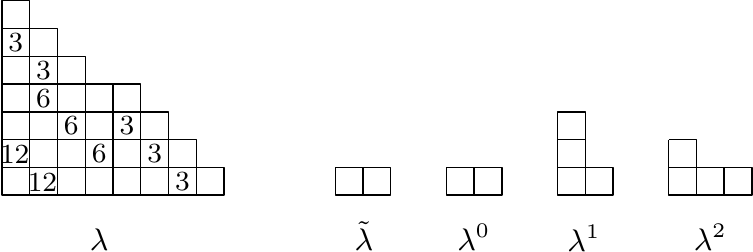}
\caption{\label{fig7}Littlewood decomposition of $\lambda=(8,7,6,5,3,2,1)$ and $t=3$. In $\lambda$, we only write the hook lengths which are integral multiples of $3$.}
\end{figure}

The vector of partitions $(\lambda^0, \lambda^1, \ldots, \lambda^{t-1})$ is usually called the \emph{$t$-quotient} of $\lambda$. Let us describe the bijection $\Omega$. We split the canonical representation $C_\lambda=(c_i)_i$ into $t$ sections, \emph{i.e.} we form the subsequence $v^k=(c_{it+k})_i$ for each $k \in \{0, \ldots, t-1\}$. The partition $\lambda^k$ is defined as  $\psi^{-1}(v^k)$. Notice that the subsequence $v^k$ is not necessarily the canonical representation of $\lambda^k$.
The partition $\tilde{\lambda}$ can be defined equivalently as the $t$-core of $\lambda$ or as follows in terms of words. For each subsequence $v^k$, we continually replace the subword $10$ by $01$. The final resulting sequence is of the form $\cdots 000111\cdots$ and is denoted by $w^k$. The $t$-core of the partition is the partition $\tilde{\lambda}$ such that the $t$ sections of the canonical representation $C_{\tilde{\lambda}}$ are exactly $w^0,w^1, \ldots, w^{t-1}$. Properties (ii) and (iii) can be derived from the following fact: each box of $\lambda$ is in one-to-one correspondence with the ordered pair of integers $(i,j)$ such that $i<j$, $c_i=1$ and $c_j=0$. Moreover the hook length of that box is equal to $j-i$.

\subsection{Littlewood decomposition of doubled distinct partitions}\label{section4.2}
Let $t$ be a positive integer.
It is already known (see \cite{GKS}) that the restriction of $\Omega$ to the set of doubled distinct partitions is a bijection with the set of vectors ($\tilde{\lambda}, \lambda^0, \lambda^1,\ldots,\lambda^{t-1}) \in DD_{(t)} \times DD \times \mathcal{P}^{t-1}$ such that $\lambda^{t-i} = \lambda^{i*}$ for $1 \leq i \leq t-1$ (where $\lambda^{i*}$ is the conjugate of $\lambda^i$).
This property can be checked on Figure~\ref{fig7}, for $\lambda=(8,7,6,5,3,2,1)\in DD$ and $t=3$.

 To prove our generalisation of Theorem~\ref{theoremeintro}, we will need new properties of the Littlewood decomposition for doubled distinct partitions.

\begin{lemma}\label{littlewood}
Let $t= 2t'+1$ be an odd positive integer, let $\lambda $ be a doubled distinct partition, and let $(\tilde{\lambda}, \lambda^0, \lambda^1,\ldots \lambda^{t-1})$ be its image under $\Omega$. The following properties hold:
\begin{itemize}
\item[(i)] $\delta_ \lambda= \delta_{\tilde{\lambda}} \delta_{\lambda^0}$.
\item[(ii)] Let $v$ be a box of $\lambda^0$ and let $V$ be its canonically associated box in $\lambda$. The box $v$ is strictly above the principal diagonal of $\lambda^0$ if and only if $V$ is strictly above the principal diagonal of $\lambda$.

\item[(iii)] Let $v=(j,k)$ be a box of $\lambda^i$, with $1\leq i\leq t'$. Denote by $v^*=(k,j)$ the box of $\lambda^{2t'+1-i}=\lambda^{i*}$ and by $V$ and $V^*$ the boxes of $\lambda$ canonically associated to $v$ and $v^*$. If $V$ is strictly above (respectively below) the principal diagonal in $\lambda^i$, then $V^*$ is strictly below (respectively above) the principal diagonal in $\lambda$.

\end{itemize}
\end{lemma}

\begin{figure} [h!]
\includegraphics[scale=1.2]{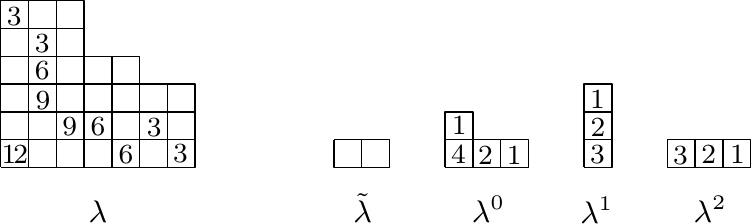}
\caption{\label{fig11}Lemma~\ref{littlewood} illustrated for $\lambda= (7,7,7,5,3,3)$ and $t=3$. The box with hook length $12$ in $\lambda$ corresponds to the box with hook length $4$ in $\lambda^0$. The boxes with hook lengths $9$ in $\lambda$ correspond to the boxes with hook lengths $3$ in $\lambda^1$ and $\lambda^2$.}
\end{figure}

\begin{proof}
Let $\lambda $ be a doubled distinct partition, and let $(\tilde{\lambda}, \lambda^0, \lambda^1,\ldots, \lambda^{t-1})$ be its image under $\Omega$. Set $\tilde{v}:= \psi({\tilde{\lambda}})$,  and $v^i := \psi(\lambda^i)$ for $0\leq i \leq t-1$ where we recall that $\psi$ is defined in \eqref{defpsi}.

We prove (i) by induction on the number of boxes in the principal diagonal of $\lambda$. It is true if $\lambda$ is empty. If $\lambda$ is non-empty, we denote by $\lambda'$ the doubled distinct partition obtained by deleting the largest principal hook length of $\lambda$. Set $\Omega(\lambda'):=(\tilde{\lambda'}, \lambda^{'0 }, \lambda^{'1 },\ldots, \lambda^{'t-1} )$,  $\tilde{v'}:= \psi(\tilde{\lambda'})$, and $ v'^i:=\psi(\lambda^{'i})$ for $0\leq i \leq t-1$. By the induction hypothesis, we have $\delta_ {\lambda'}= \delta_{\tilde{\lambda'}} \delta_{\lambda^{'0}}$. Deleting the largest principal hook length of $\lambda$ corresponds, in terms of words, to turn the first $1$ of $\psi(\lambda)$ into a $0$, and the last $0$ of $\psi(\lambda)$ into a $1$.
Two cases can occur.

{\bf Case 1:} the first $1$ in $\psi(\lambda)$ belongs to $v^0$ (see the example in Figure~\ref{fig13} below). In this case, there are exactly $kt-1$ letters between this $1$ and the symbol ``.", and as $\lambda$ is a doubled distinct partition, there are exactly $kt$ letters between the symbol ``." and the last letter $0$. So the last $0$ belongs also to $v^0$. Turning the first $1$ into a $0$ and the last $0$ into a $1$ actually changes only $v^0$ and deletes the largest principal hook of $\lambda^0$. The $t$-core $\tilde{\lambda}$ does not change. So $\lambda^{'0}$ is equal to the partition $\lambda^0$ in which we delete the largest principal hook and $\delta_\lambda=-\delta_{\lambda'}=-\delta_{\tilde{\lambda'}} \delta_{\lambda^{'0}}=\delta_{\tilde{\lambda}} \delta_{\lambda^0}$.

\begin{figure}[!h]
\includegraphics[scale=1.2]{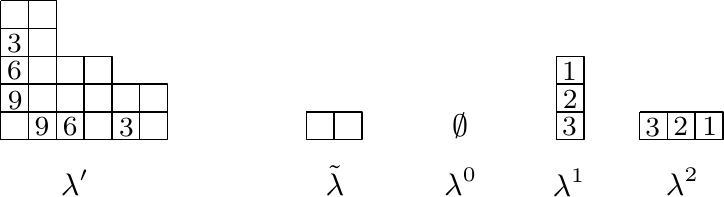}
\vspace*{0.5cm}

\includegraphics[scale=1]{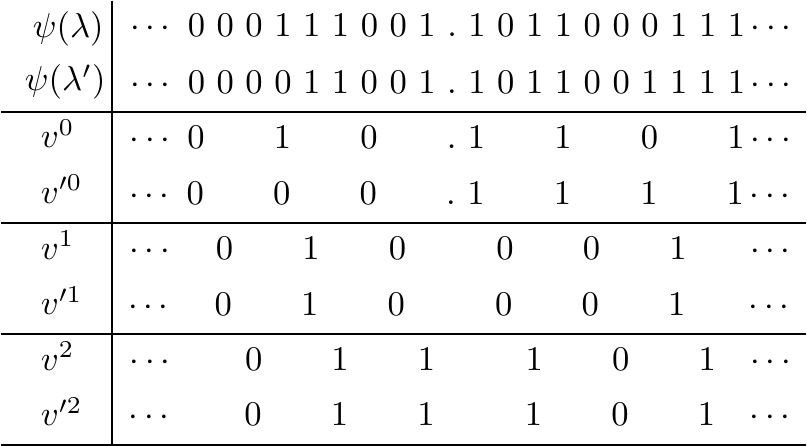}
\caption{\label{fig13}Illustration of the proof of (i) Case 1, for $\lambda=(7,7,7,5,3,3)$ as in Figure~\ref{fig11}.}
\end{figure}

{\bf Case 2:} the first $1$ in $\psi(\lambda)$ belongs to a word $v^i$ (with $1\leq i \leq t$) and the last $0$ in $\psi(\lambda)$ belongs to the word $v^{t-i}$. Turning the first $1$ into a $0$ deletes the first column in $\lambda^i$ and turning the last $0$ into a $1$ deletes the first row in $\lambda^{t-i}$. The partition $\lambda^0$ does not change, so $\delta_{\lambda^{0}}=\delta_{\lambda^{'0}}$. We want actually to prove that the parity of the Durfee square of $\tilde{\lambda}$ is different from the parity of the one of $\tilde{\lambda'}$. By continually replacing the subword 10 by 01 in $v^i$ (respectively $v^{t-i}$), we obtain a subword of the form $\cdots 000111\cdots$, where the last $0$ is in position $k_1$ (respectively $k_2$), where $k_1$ and $k_2$ are integers.
By continually replacing the subword 10 by 01 in $v'^i$ (respectively $v'^{t-i}$), we also obtain a subword of the form $\cdots 000111\cdots$, but here the last $0$ is in position $k_1+1$ (respectively $k_2-1$). According to the respective distances of $k_1$ and $k_2$ to the ``." symbol in $\tilde {v}$, the Durfee square of $\tilde{\lambda'}$  increases or decreases by $1$ the Durfee square of $\tilde{\lambda}$. So $\delta_\lambda=-\delta_{\tilde{\lambda'}}\delta_{\lambda'^0}=\delta_{\tilde{\lambda}
}\delta_{\lambda^0}$.
\medskip

We now prove (ii). Let $v=(j,k)$ be a box of $\lambda^0$ and set $V$ its canonically associated box in $\lambda$. If $\lambda^0$ is empty, the property is trivial. Otherwise, as $\lambda^0$ is a doubled distinct partition, we can write $v^0= \cdots 00w.1 f(w)11\cdots$, where $w$ is a word beginning by $1$ (recall that $f$ is described in Definition~\ref{definitionf}). Assume that $v$ is strictly above the principal diagonal of $\lambda^0$ and in its Durfee square (the other cases are symmetric). Using the one-to-one correspondence between boxes and ordered pairs of integers, we can decompose $w$ in the following way: $w=w_11w_2$ where there are exactly $k-1$ occurrences of $1$ in $w_1$. We can also decompose $f(w)$ in the following way: $f(w)=w_30w_4$, where there are exactly $j-1$ occurrences of $0$ in $w_4$. As $\lambda^0$ is a doubled distinct partition and as $j>k$,  $f(w_1)$ is a suffix of $w_4$. This situation also holds in the canonical representation of $\psi(\lambda)$. We can write $\psi(\lambda)=\cdots 00u_11u_2.1u_30u_411\cdots$, where the pair of the first $1$ after $u_1$ and the first $0$ after $u_3$ corresponds to the box $V$. The word $f(u_1)$ is a suffix of $u_4$, and $u_1$ contains strictly less $1$'s than the number of $0$'s in $u_4$, so the box $V$ is strictly above the principal diagonal of $\lambda$.

\medskip

We now prove (iii). Notice first that for a reason of parity, the boxes $V$ and $V^*$ can not belong to the principal diagonal of $\lambda$. Assume without loss of generality that $k <j$ (and so $V$ is strictly above the principal diagonal in $\lambda$) . We can write $v^i=\cdots00w_11w_20w_311\cdots$, where $w_1$ begins by $1$, $ w_3$ ends by $0$, and there are $k-1$ occurrences of $1$ in $w_1$ and $j-1$ occurrences of $0$ in $w_3$ (where the $1$ and the $0$ before and after $w_2$ correspond to the box $v$). As $\lambda^{2t'+1-i}=\lambda^{i*}$, we can write $v^{t-i}=f(v^i)=\cdots00u_11u_20u_311\cdots$ where $u_1$ begins by $1$, $ u_3$ ends by $0$, and there are $j-1$ occurrences of $1$ in $u_1$ and $k-1$ occurrences of $0$ in $u_3$. Here, the $1$ and the $0$ before and after $u_2$ correspond to the box $v^*$. These informations allow us to describe $\psi(\lambda)$ in the following way:
$ \psi(\lambda)= \cdots00x_11x_21x_30x_40x_511\cdots$, where the $1$ after $x_1$ comes from the $k^{th}$ occurrence of $1$ in $v^i$, the $1$ after $x_2$ comes from the $i^{th}$ occurrence of $1$ in $v^{t-i}$, the $0$ after $x_3$ comes from the $i^{th}$ occurrence of $0$ reading $v^{i}$ from right to left and the $0$ after $x_4$ comes from the $k^{th}$ occurrence of $0$ reading $v^{t-i}$ from right to left. The box $V$ corresponds to the ordered pair given by the $1$ after $x_1$ and the $0$ after $x_3$. Therefore there is at least one more occurrence of $0$ in $x_40x_5$ than the number of occurrences of $1$ in $x_1$, so the box $V$ is strictly above the principal diagonal in $\lambda$. The box $V^*$ corresponds to the ordered pair given by the $1$ after $x_2$ and the $0$ after $x_4$, so $V^*$ is strictly below the principal diagonal in $\lambda$.

\end{proof}

Property (i) of the previous lemma allows us to compute the following signed generating function of doubled distinct $t$-cores, which is surprisingly simpler than the unsigned one given in the proof of Proposition~\ref{propgeneratingfunction}, or in \cite{GKS}.

\begin{lemma}\label{seriegen}
Let $t=2t'+1$ be an odd positive integer. The following equality holds:
\begin{equation}
\sum_{\tilde{\lambda} \in DD_{(t)}} \delta_{\tilde{\lambda}}\, x^{|\tilde{\lambda}|/2}= \prod_{k \geq 1} (1-x ^k) (1-x^{kt})^{t'-1}.
\end{equation}

\end{lemma}

\begin{proof}
 We can use the Littlewood decomposition of doubled distinct partitions and Lemma~\ref{littlewood}~(i) to obtain:
 \begin{equation}
\sum_{\lambda \in DD} \delta_{\lambda}\, x^{|\lambda|/2}=\sum_{\tilde{\lambda} \in DD_{(t)}} \delta_{\tilde{\lambda}}\, x^{|\tilde{\lambda}|/2} \times \sum_{\lambda^0 \in DD} \delta_{\lambda^0}\, x^{t|\lambda^0|/2} \times \left( \sum_{\lambda \in \mathcal{P}} x ^{t|\lambda|} \right)^{t'}. \label{proofseriegen}
 \end{equation}
Denoting by $\mathcal{P}_{D}$ the set of partitions with distinct parts, we also have:
\begin{equation*}
\sum_{\lambda \in DD} \delta_{\lambda}\,x^{|\lambda|/2}= \sum_{{\lambda \in \mathcal{P}_{D}}}(-1)^{\#\{\scriptsize{parts~of~}\lambda\}}x^{|\lambda|}= \prod_{k \geq 1} \left(1-x^k\right).
\end{equation*}
Then \eqref{proofseriegen} can be rewritten:
\begin{equation*}
\prod_{k \geq 1} \left(1-x^k\right)=\left(\sum_{\tilde{\lambda} \in DD_{(t)}} \delta_{\tilde{\lambda}}\, x^{|\tilde{\lambda}|/2}\right) \times \prod_{k \geq 1} \left(1-x^{kt}\right)\times \left( \prod_{k \geq 1} \frac{1}{1-x^{kt}}\right)^{t'},
\end{equation*}
which proves the lemma.
\end{proof}

\subsection{Generalization of Theorem~\ref{theoremeintro}}\label{section4.3}
We prove here Theorem~\ref{generalisation}, which will be seen to generalize Theorem~\ref{theoremeintro}, and we derive several applications.

\begin{proof}[Proof of Theorem~\ref{generalisation}.]
Let $\lambda$ be a doubled distinct partition. We  first transform the expression \begin{equation}\label{eqrhs}\displaystyle \delta_{\lambda} x^{|\lambda|/2} \prod_{ h \in \mathcal{H}_t(\lambda)}\left( y -\frac{yt(2z+2)}{\varepsilon_h ~ h}\right)\end{equation} by using the Littlewood decomposition of $\lambda$. Set $\Omega(\lambda)=(\tilde{\lambda}, \lambda^0, \lambda^1,\ldots, \lambda^{t-1})$. According to Lemma~\ref{littlewood}~(i), we have $\delta_ \lambda= \delta_{\tilde{\lambda}} \delta_{\lambda^0}$.

Let $B_i$ be the multi-set of hook lengths in $\mathcal{H}_{t}(\lambda)$ coming from the boxes of $\lambda$ which correspond to the ones of $\lambda^i$, for $0\leq i\leq t-1$. 
According to property (iii) of the Littlewood decomposition and Lemma~\ref{littlewood}~(ii), we have:
\begin{equation}
\prod_{ h \in B_0} \left(y -\frac{yt(2z+2)}{\varepsilon_h ~ h}\right)= \prod_{h \in \mathcal{H}(\lambda^0)} \left(y-\frac{y(2z+2)}{\varepsilon_h ~h} \right).\label{eq47}
\end{equation}
 
 Let $v=(j,k)$ be a box of $\lambda^i$, with $1\leq i\leq t'$. Denote by $v^*=(k,j)$ the box of $\lambda^{2t'+1-i}=\lambda^{i*}$ and by $V$ and $V^*$ the boxes of $\lambda$ canonically associated to $v$ and $v^*$. 
By Lemma~\ref{littlewood}~(iii), one of the two boxes $V$ and $V^*$ is strictly below the principal diagonal of $\lambda$, and the other is strictly above the principal diagonal of $\lambda$. So we have:
\begin{equation*}\label{eqhV}
\left(y -\frac{yt(2z+2)}{\varepsilon_{h_V} ~ h_V}\right)\left(y -\frac{yt(2z+2)}{\varepsilon_{h_ {V^*}} ~ h_{V^*}}\right)=y^2 -\left(\frac{yt(2z+2)}{ ~ h_V}\right)^2= y^2 -\left(\frac{y(2z+2)}{ h_v}\right)^2,
\end{equation*}
where the last equality follows from $h_V= th_v$ according to property (iii) of the Littlewood decomposition.
Multiplying this over all boxes $V$ in $B_i$ gives:
\begin{equation}
\prod_{ h \in B_i\cup B_{t-i}} \left(y -\frac{yt(2z+2)}{\varepsilon_h ~ h}\right)= \prod_{ h \in \mathcal{H}(\lambda ^ i)} \left(y^2 -\left(\frac{y(2z+2)}{h}\right)^2 \right).\label{eq48}
\end{equation}
Using \eqref{eq47}, \eqref{eq48}, and property (ii) of the Littlewood decomposition,   we can rewrite \eqref{eqrhs} as follows:
\begin{multline*}
\delta_{\lambda} x^{|\lambda|/2} \prod_{ h \in \mathcal{H}_t(\lambda)}\left( y -\frac{yt(2z+2)}{\varepsilon_h ~ h}\right)=\delta_{\tilde{\lambda}} x^{|\tilde{\lambda}|/2} \\\times\delta_{\lambda ^0} x^{t|\lambda^0|/2} \prod_{h \in \mathcal{H}(\lambda^0)} \left(y-\frac{y(2z+2)}{\varepsilon_h ~h} \right)\\\times \prod_{i=1}^{t'}  x^{t|\lambda^i|} \prod_{h \in \mathcal{H}(\lambda^i)} \left(y^2 -\left(\frac{y(2z+2)}{h}\right)^2 \right).
\end{multline*}
We sum this over all doubled distinct partitions. The left-hand side becomes the one of \eqref{eqgeneralisation}, while the right-hand side can be written as a product of three terms. The first one is 
\begin{equation*}\sum_{\tilde\lambda \in DD_{(t)}}\delta_{\tilde{\lambda}} x^{|\tilde{\lambda}|/2}=\prod_{k \geq 1} (1-x ^k) (1-x^{kt})^{t'-1}
\end{equation*} 
by Lemma~\ref{seriegen}, while the second one is 
\begin{equation*}\sum_{\lambda^0 \in DD} \delta_{\lambda ^0} x^{t|\lambda^0|/2} \prod_{h \in \mathcal{H}(\lambda^0)} \left(y-\frac{y(2z+2)}{\varepsilon_h ~h} \right)= \prod_{k \geq 1} \left(1-x^{kt}y^{2k}\right)^{2z^2+z}
 \end{equation*}
  by Theorem~\ref{theoremeintro} applied with $t$ replaced by $z$ and $x$ replaced by $x^ty^2$. 

Finally, the third term is 
\begin{equation*} \left(\sum_{\lambda \in \mathcal{P}}  x^{t|\lambda|} \prod_{h \in \mathcal{H}(\lambda)} \left(y^2 -\left(\frac{y(2z+2)}{h}\right)^2 \right)\right)^{t'} = \left(\prod_{k \geq 1} \left(1-x^{kt}y^{2k}\right)^{(2z+2)^2-1}\right)^{t'}
\end{equation*}
 by the classical Nekrasov--Okounkov formula~\eqref{nekrasov} applied with $x$ replaced by $ x^t y^2$ and $z$ by $(2z+2)^2$.

It only remains to check that $2z^2+z+t'\left((2z+2)^2-1\right)= (2z+1)(zt+3(t-1)/2)$ to finish the proof.
\end{proof}

Notice that when $y=t=1$ we recover Theorem~\ref{theoremeintro}, and when $y=0$, we recover Lemma~\ref{seriegen}. So Theorem~\ref{generalisation} unifies the Macdonald identities generalized by Theorem~\ref{theoremeintro} and the generating function of doubled distinct $t$-cores $\lambda$ with weight $\delta_\lambda$.

Next, we list some consequences of Theorem~\ref{generalisation} on doubled distinct partitions, in the same vein as Han did for partitions \cite{HAN}. In all what follows, $t=2t'+1$ is an odd integer.

\begin{corollary}[$z=-1$ in Theorem~\ref{generalisation}] We have:
\begin{equation}
\sum_{\lambda \in DD} \delta_{\lambda}\, x^{|\lambda|/2} y^{\# \mathcal{H}_t (\lambda)} = \prod_{k \geq 1} (1-x ^k)(1-x^{kt})^ {t'-1} (1-(x ^ty^ 2)^ k)^{1-t'}.
\end{equation}
\end{corollary}

\begin{corollary}[$z=-1,~ y=1$ in Theorem~\ref{generalisation}] We have:
\begin{equation}\label{eq42}\sum_{\lambda \in DD} \delta_\lambda\, x^{|\lambda|/2} = \prod_{k \geq 1} (1-x^k).
\end{equation}
\end{corollary}
Note that by definition of $DD$ and $\delta_\lambda$, \eqref{eq42} is equivalent to \eqref{partpartdistinct}.

\begin{corollary}[$z=-1$, $y=\sqrt{-1}$ in Theorem~\ref{generalisation}] We have:
\begin{equation}
\sum_{\lambda \in DD} \delta_\lambda\, x^{|\lambda|/2} (-1)^{\# \mathcal{H}_t(\lambda)/2} = \prod_{k \geq 1} (1-x^k)\prod_{k\mbox{\scriptsize \,odd\,}\geq 1} \left(\frac{1-x^{kt}}{1+x^{kt}}\right)^{t'-1}.
\end{equation}

\end{corollary}

\begin{proof}
By setting $z=-1$ and $y=\sqrt{-1}$ in Theorem~\ref{generalisation},
we obtain:
\begin{eqnarray*}
\sum_{\lambda \in DD} 
\delta_\lambda\, x^{|\lambda|/2} (-1)^{\# \mathcal{H}_t(\lambda)/2} &= \displaystyle\prod_{k \geq 1} (1-x^k)(1-x^{kt})^{t'-1}(1-(-1)^k x^{kt})^{1-t'}
\\&=\frac{\displaystyle\prod_{k \geq 1} (1-x^k)(1-x^{kt})^{t'-1} }{\displaystyle\prod_{k\mbox{\scriptsize \,even}\,\geq 2} (1-x^{kt})^{t'-1} \displaystyle\prod_{k\mbox{\scriptsize \,odd}\,\geq 1} (1+x^{kt})^{t'-1}}
\\&=\displaystyle\prod_{k \geq 1} (1-x^k)\displaystyle\prod_{ k\mbox{\scriptsize \,odd}\,\geq 1} \left(\frac{1-x^{kt}}{1+x^{kt}}\right)^{t'-1}.\hspace*{1cm}
\end{eqnarray*}
\end{proof}

\begin{corollary}[$y=1$ in Theorem~\ref{generalisation}]\label{y=1} We have:
\begin{equation}
\sum_{\lambda \in DD} \delta_{\lambda}\, x^{|\lambda|/2} \prod_{ h \in \mathcal{H}_t(\lambda)} \left(1 -\frac{t(2z+2)}{\varepsilon_h ~ h}\right)= \prod_{k \geq 1}  (1-x ^k)(1-x^{kt})^ {(z+1)(2zt+2t-3)} .
\end{equation}
\end{corollary}

\begin{corollary}[$y=1$ and compare the coefficients of $z+1$ in Theorem~\ref{generalisation}] We have:
\begin{equation}
\sum_{\lambda \in DD} \delta_\lambda\, x^{|\lambda|/2}\sum_{\stackrel{h \in \mathcal{H}_t(\lambda)}{h \in \Delta}}\frac{1}{h}=\frac{-1}{2t}\prod_{k \geq 1} (1-x^k) \sum_{m \geq 1} \frac{x^{tm}}{m(1-x^{tm})},
\end{equation}
where we recall that $\Delta$ is the set of principal hook lengths of $\lambda$.
\end{corollary}

\begin{corollary}[$2z+2 =-b/y$, $y \rightarrow 0$ in Theorem~\ref{generalisation}] \label{cor6} We have:
\begin{multline}
\sum_{\lambda \in DD}  \delta_\lambda\, x^{|\lambda|/2} \prod_{ h \in \mathcal{H}_t(\lambda)} \frac{bt}{ h\, \varepsilon_h} =\exp(-tb^2 x^t
/2) \prod_{k \geq 1}  (1-x ^k)(1-x^{kt})^ {t'-1}.
\label{eqexp}\end{multline}
\end{corollary}

\begin{proof}
By setting $2z+2 =-b/y$, and $y \rightarrow 0$ in Theorem~\ref{generalisation}, the left-hand side becomes exactly the one of \eqref{eqexp}. 

On the right-hand side, the term $\displaystyle\prod_{k \geq 1} (1-x ^k)(1-x^{kt})^ {t'-1}$ does not change.
It remains to examine $\displaystyle\prod_{k \geq 1} \left(1-(x^ty^2)^k \right)^{2tz^2+4zt-3z+3t'}$. As $y \rightarrow 0$, we have: \begin{equation*}\lim_{y \to 0} \displaystyle\prod_{k \geq 1} \left(1-(x^ty^2)^k \right)^{3t'}=1.\end{equation*}
By using \eqref{eqex}, we can write:
\begin{eqnarray*}
\prod_{k \geq 1} \left(1-(x^ty^2)^k \right)^{2tz^2+4zt-3z} =&\exp\left( -(2tz^2+4zt-3z) \displaystyle\sum_{k \geq 1} \frac{(x^ty^2)^k}{k(1-(x^ty^2))^k}\right) \\ =&\exp \left(\displaystyle-\frac{tb ^2} {2} \left(x^t+ O(y^2)\right)\right)\hspace{3cm}\\ \underset{y \to 0}{\longrightarrow}&\exp\left(\displaystyle-\frac{tb ^2x^t}{2}\right),\hspace*{5cm}
\end{eqnarray*}
where $O(y^2)$ is a term which tends to $0$ when $y \rightarrow 0$, and the corollary follows.
\end{proof}

\begin{corollary}[$t=1$ in Corollary~\ref{cor6}]
\begin{equation}
\sum_{\lambda \in DD} x^{|\lambda|/2} \prod_{h \in \mathcal{H}(\lambda)} \frac{b}{h} = \exp(b ^2 x/2).
\end{equation}

\end{corollary}
\begin{proof}
After setting $t=1$ in Corollary~\ref{cor6}, we can notice that 
\begin{equation}\label{eqsign}
\displaystyle\delta_\lambda \prod_{h \in  \mathcal{H}(\lambda)} \varepsilon_h = (-1) ^{|\lambda|/2}.
\end{equation}
The result follows when we replace $x$ by $-x$.
\end{proof}

\begin{corollary}\label{4.11}[compare the coefficients of $b^ {2n} x^ {tn}$ in Corollary~\ref{cor6}] \label{cor9} We have for any nonnegative integer $n$:
\begin{equation}
\sum_{\stackrel{ \lambda \in DD, ~ |\lambda|=2t n}{\# \mathcal{H}_t(\lambda)=2n }}\displaystyle\;\delta_\lambda \prod _{h \in \mathcal{H}_t (\lambda)} \frac{1}{ h\, \varepsilon_h}=\displaystyle \frac{(-1)^n}{n! t^n 2^n}.
\end{equation}
\end{corollary}
The condition on the sum can also be written $\lambda \in DD, ~ |\lambda|=2t n$ and the $t$-core of $\lambda$ is empty.
When $t=1$, we get back the symplectic hook formula \eqref{hookf} thanks to \eqref{eqsign}.

\begin{corollary}\label{4.12}[compare the coefficients of $b^ {2n} x^ {tn+m}$ in Corollary~\ref{cor6}] We have for any nonnegative integers $n$ and $m$:
\begin{equation}
\sum_{\stackrel{ \lambda \in DD, ~ |\lambda|=2t n+m}{\# \mathcal{H}_t(\lambda)=2n }}\;\displaystyle \delta_\lambda \prod _{h \in \mathcal{H}_t (\lambda)} \frac{1}{ h\, \varepsilon_h}=\displaystyle \frac{(-1)^nc_t(m)}{n! t^n 2^n},
\end{equation}
where $c_t(m)$ is the number of $DD$ $t$-cores of weight $m$.
\end{corollary}
The condition on the sum can also be writen $\lambda \in DD, ~ |\lambda|=2t n$ and the weight of the $t$-core of $\lambda$ is $m$.

\begin{corollary}\label{4.13}[compare the coefficients of $x^{tn}y^{2n}(z+1)^{2n-1}$ in Theorem~\ref{generalisation}]\label{eqpointee} We have:
\begin{equation}
\sum_{\stackrel{\lambda \in DD, |\lambda|=2tn}{\#\mathcal{H}_t(\lambda)=2n}}\,\delta_\lambda\prod_{h \in \mathcal{H}_t
(\lambda)}\frac{1}{h \, \varepsilon_h}\sum_{h \in \mathcal{H}_t
(\lambda)}h \,\varepsilon_h = \frac{3(-1)^n}{(n-1)!t^n 2^{n}}.\end{equation}
\end{corollary}
This formula is the analogue for doubled distinct partitions of the \emph{marked hook formula} (\cite[Corollary 5.7]{HAN}). When $t=1$, Corollary~\ref{eqpointee} reduces to
\begin{equation*}\sum_{\lambda \in DD, |\lambda|=2n}\,\delta_\lambda\prod_{h \in \mathcal{H}
(\lambda)}\frac{1}{h \, \varepsilon_h}\sum_{h \in \mathcal{H}
(\lambda)}h \,\varepsilon_h = \frac{3(-1)^n}{(n-1)! 2^{n}},
\end{equation*}
which is in fact equivalent to the symplectic hook formula \eqref{hookf}, thanks to \eqref{eqsign} and after noticing that $\displaystyle \sum_{h \in \mathcal{H}
(\lambda)}h \,\varepsilon_h =3n$ for any $\lambda \in DD$ with $|\lambda| =2n$.

\section{Concluding remarks and questions}\label{section5}

A natural question which arises through our work and Remark~\ref{king} is whether we can prove combinatorially, starting from Macdonald identities, part or all of the equations $(5.8a)$ to $(5.8j)$ of \cite{KING89} (specialized at $x_i=1$). Apart from the already mentioned equation $(5.8b)$, a particular case of this problem, namely equation $(5.8a)$, corresponds to the type $\widetilde{D}$, and would give us a combinatorial interpretation of the character formula in this type.

Another problem which would arise after obtaining such combinatorial interpretations, would be to generalize the equations $(5.8a)$--$(5.8j)$ of \cite{KING89} by using the Littlewood decomposition. This would probably involve other interesting families of integer partitions and new properties of the corresponding Littlewood decompositions.

Surprisingly, when one aims to obtain a generalization of Theorem~\ref{theoremeintro} analogous to Theorem~\ref{generalisation} for $t$ even, some new affine types appear naturally. Indeed, to prove such a generalization, we need to compute the following weighted generating series for self conjugate partitions, which is a generalization of Macdonald identity in type $\widetilde{C}\check~$ (see \cite[p. 137, (b)]{ARS}), and which states that for any complex number $z$, the following expansion holds:
\begin{equation}\label{equaccheck}
\left(\prod_{i \geq 1} \frac{(1-x^{2i})^{z+1}}{1-x^i}\right)^{2z-1}= \sum_{\lambda \in SC}\delta_\lambda \, x^{|\lambda|} \prod_{h \in \mathcal{H}(\lambda)}\left(1-\frac{2z}{h\, \varepsilon_h}\right),
\end{equation}
where we recall that $SC$ is the set of selfconjugate partitions. Again, note that this is related to King's work (see \cite[Equation $(5.8j)$]{KING89}). Moreover, we are able to prove combinatorially \eqref{equaccheck} and to generalize it through new properties of the Littlewood decomposition. This will be done in a forthcoming work.

From the connection between our combinatorial approach and King's algebraic point of view, one can also wonder whether this could be lifted to the level of characters of Kac--Moody algebras, and therefore connected to the recent works of Bartlett--Warnaar \cite{BAR2} and Rains--Warnaar \cite{RW}.

\medskip

In \cite{DH}, Han and Dehaye generalized the Nekrasov-Okounkov formula \eqref{nekrasov} in another way, by introducing a weight function $\tau$. It is therefore natural to try to do the same with our formula \eqref{eqtheoremeintro} in Theorem~\ref{theoremeintro}.

One can also ask whether there is a proof of \eqref{eqtheoremeintro} as in \cite{IEA}, \textit{i.e.} through a symplectic Cauchy formula and an analogue of the topological vertex.

In a more number theoretical point of view, one can wonder if we can deduce from our work some congruences on the number of doubled distinct partitions, as it was done from Han formulas in \cite{CW, KEI} .

Finally, inspired by the marked hook formula (see \cite[Theorem 2.3]{HAN09}), Stanley proved in \cite{STA09} that if F is any symmetric function and if 
\begin{equation*}
\Phi_n (F) := \sum_{|\lambda|=n} \left(\displaystyle\prod_{h \in \mathcal{H}(\lambda)} h^ {-1}\right) F(h^2 : h \in \mathcal{H}(\lambda)),
\end{equation*}
then $\Phi_n (F)$ is a polynomial function of $n$. In the same way, inspired by Corollaries~\ref{4.11}--\ref{4.13}, one may wonder if such a property involving the doubled distinct partitions and the statistics $h\varepsilon_h$ and $\delta_\lambda$ could  hold.

\medskip
{\bf Acknowledgements.}
We would like to thank Guo-Niu Han, Ronald King and Bruce Westbury for helpful comments and suggestions on an earlier version of this paper.

\bibliographystyle{plain}
\bibliography{bibliographieang}

\begin{thebibliography}{10}

\bibitem{BAR2}
Nick Bartlett and S.~Ole Warnaar.
\newblock Hall-littlewood polynomials and characters of affine lie algebras.
\newblock {\em Preprint arXiv:1304.1602}, 2013.

\bibitem{CW}
Dan Collins and Sally Wolfe.
\newblock Congruences for {H}an's generating function.
\newblock {\em Involve}, 2009.

\bibitem{DH}
Paul-Olivier Dehaye and Guo-Niu Han.
\newblock A multiset hook length formula and some applications.
\newblock {\em Discrete Math.}, 311(23):2690--2702, 2011.

\bibitem{KS}
N.~El~Samra and Ronald~C. King.
\newblock Dimensions of irreducible representations of the classical {L}ie
  groups.
\newblock {\em J. Phys. A}, 12(12):2317--2328, 1979.

\bibitem{GKS}
Frank Garvan, Dongsu Kim, and Dennis Stanton.
\newblock Cranks and t-cores.
\newblock {\em Invent. Math.}, 101(1):1--17, 1990.

\bibitem{HAN09}
Guo-Niu Han.
\newblock Some conjectures and open problems on partition hook lengths.
\newblock {\em Exp. Math.}, 18(1):97--106, 2009.

\bibitem{HAN10b}
Guo-Niu Han.
\newblock Hook lengths and shifted parts of partitions.
\newblock {\em Ramanujan J.}, 23(1-3):127--135, 2010.

\bibitem{HAN}
Guo-Niu Han.
\newblock The {N}ekrasov-{O}kounkov hook length formula: refinement, elementary
  proof, extension and applications.
\newblock {\em Ann. Inst. Fourier}, 60:1--29, 2010.

\bibitem{HJ}
Guo-Niu Han and Kathy Ji.
\newblock Combining hook length formulas and {BG}-ranks for partitions via the
  {L}ittlewood decomposition.
\newblock {\em Trans. Amer. Math. Soc.}, 363(2):1041--1060, 2011.

\bibitem{IEA}
Amer Iqbal, Shaheen Nazir, Zahid Raza, and Zain Saleem.
\newblock Generalizations of {N}ekrasov-{O}kounkov identity.
\newblock {\em Ann. Comb.}, 16(4):745--753, 2012.

\bibitem{GK}
Gordon James and Adalbert Kerber.
\newblock {\em The representation theory of the symmetric group}, volume~16 of
  {\em Encyclopedia of Mathematics and its Applications}.
\newblock Addison-Wesley Publishing Co., Reading, Mass., 1981.
\newblock With a foreword by P. M. Cohn, With an introduction by Gilbert de B.
  Robinson.

\bibitem{KEI}
William~J. Keith.
\newblock Polynomial analogues of {R}amanujan congruences for {H}an's
  hooklength formula.
\newblock {\em Acta Arith.}, 160(4):303--315, 2013.

\bibitem{KING89}
Ronald~C. King.
\newblock {$S$}-functions and characters of {L}ie algebras and superalgebras.
\newblock In {\em Invariant theory and tableaux ({M}inneapolis, {MN}, 1988)},
  volume~19 of {\em IMA Vol. Math. Appl.}, pages 226--261. Springer, New York,
  1990.

\bibitem{KING15}
Ronald~C. King.
\newblock Some notes on specialisations of {M}acdonald identities.
\newblock Private communication, 2015.

\bibitem{KOS}
Bertram Kostant.
\newblock Powers of the {E}uler product and commutative subalgebras of a
  complex simple {L}ie algebra.
\newblock {\em Invent. Math.}, 158(1):181--226, 2004.

\bibitem{ARS}
Ian~G. Macdonald.
\newblock Affine root systems and {D}edekind's {$\eta $}-function.
\newblock {\em Invent. Math.}, 15:91--143, 1972.

\bibitem{NO}
Nikita~A. Nekrasov and Andrei Okounkov.
\newblock Seiberg-{W}itten theory and random partitions.
\newblock In {\em The unity of mathematics}, volume 244 of {\em Progr. Math.},
  pages 525--596. Birkh\"auser Boston, Boston, MA, 2006.

\bibitem{RW}
Eric~M. Rains and S.~Ole Warnaar.
\newblock Bounded littlewood identities.
\newblock {\em Preprint ar{X}iv:1506.02755}, 2015.

\bibitem{JPS}
Jean-Pierre Serre.
\newblock {\em Cours d'arithm\'etique}, volume~2 of {\em Collection SUP: ``Le
  Math\'ematicien''}.
\newblock Presses Universitaires de France, Paris, 1970.

\bibitem{EC}
Richard~P. Stanley.
\newblock {\em Enumerative combinatorics. {V}ol. 2}, volume~62 of {\em
  Cambridge Studies in Advanced Mathematics}.
\newblock Cambridge University Press, Cambridge, 1999.
\newblock With a foreword by Gian-Carlo Rota and appendix 1 by Sergey Fomin.

\bibitem{STA09}
Richard~P. Stanley.
\newblock Some combinatorial properties of hook lengths, contents, and parts of
  partitions.
\newblock {\em Ramanujan J.}, 23(1-3):91--105, 2010.

\bibitem{EC1}
Richard~P. Stanley.
\newblock {\em Enumerative combinatorics. {V}olume 1}, volume~49 of {\em
  Cambridge Studies in Advanced Mathematics}.
\newblock Cambridge University Press, Cambridge, second edition, 2012.

\bibitem{VEL}
Ameya Velingker.
\newblock On an exact formula formula for the coefficients of {H}an's
  generating function.
\newblock {\em To appear in Ann. Comb.}, 2009.

\bibitem{WEST}
Bruce~W. Westbury.
\newblock Universal characters from the {M}acdonald identities.
\newblock {\em Adv. Math.}, 202(1):50--63, 2006.

\end{thebibliography}

\end{document}